\documentclass[a4paper,oneside]{amsart}

\usepackage{amssymb,amsfonts,enumerate,color,comment,algpseudocode,algorithm,tikz,thmtools}
\usepackage[colorlinks]{hyperref} 

\usetikzlibrary{shapes.geometric,arrows.meta}
\tikzset{
    buffer/.style={
        draw,
        shape border rotate=90,
        isosceles triangle,
        isosceles triangle apex angle=60,
        fill=white,
        node distance=1.5cm,
        minimum height=3em
    }
}

\makeatletter
\DeclareRobustCommand{\rvdots}{%
  \vbox{
    \baselineskip2\p@\lineskiplimit\z@
    \kern-\p@
    \hbox{.}\hbox{.}\hbox{.}
  }}
\makeatother

\usepackage[left=3cm,right=3cm,top=25mm,bottom=25mm]{geometry}

\title[To reorient is easier than to orient]{To reorient is easier than to orient:\\ an on-line algorithm for reorientation of graphs}
\author[M.~Fiori-Carones]{Marta Fiori-Carones}
\address{Dipartimento di scienze matematiche, informatiche e fisiche, Universit\`a di Udine, Via delle Scienze 208, 33100 Udine --- Italy}
\email{marta.fioricarones@uniud.it}
\author[A.~Marcone]{Alberto Marcone}
\address{Dipartimento di scienze matematiche, informatiche e fisiche, Universit\`a di Udine, Via delle Scienze 208, 33100 Udine --- Italy}
\email{alberto.marcone@uniud.it}
\thanks{We thank Nicola Gigante and Paul Shafer for useful discussions about the topic of the paper.\\
Both author's research was partially supported by the departmental PRID funding \lq\lq HiWei --- The higher levels of the Weihrauch hierarchy\rq\rq.}
\date{\today}

\newtheorem{theorem}{Theorem}[section]
\newtheorem*{maintheorem}{Main Theorem}
\newtheorem{corollary}[theorem]{Corollary}
\newtheorem{lemma}[theorem]{Lemma}
\newtheorem{property}[theorem]{Property}
\theoremstyle{definition}
\newtheorem{definition}[theorem]{Definition}
\newtheorem{notation}[theorem]{Notation}
\theoremstyle{remark}  
\newtheorem{obs}[theorem]{Observation}
\newtheorem{remark}[theorem]{Remark}
\newtheorem{ex}[theorem]{Example}

 \newtheoremstyle{fatto}
  {}
  {}
  {\itshape}
  {5mm}
  {\bfseries} 
  {.}
  {2mm}
  {}
\theoremstyle{fatto}
\newtheorem{claim}{Claim}[theorem]

\declaretheoremstyle[
  spaceabove=-3pt,%
  spacebelow=6pt,%
  headfont=\normalfont\itshape,%
  postheadspace=1em,%
  qed=\footnotesize{\(\square\)}]{mystyle}
\declaretheorem[name={Proof},style=mystyle,unnumbered]{prf}

\newcommand{\nat}{\mathbb{N}}
\newcommand{\rca}{\mathsf{RCA}_0}
\newcommand{\wkl}{\mathsf{WKL}_0}
\newcommand{\imp}{\rightarrow}
\newcommand{\E}{\,E\,}
\newcommand{\R}{\,R\,}
\newcommand{\Imp}{\Rightarrow}

\newcommand{\Biimp}{\Leftrightarrow}
\newcommand{\pt}[1]{\ensuremath{\mathrm{pt}(#1)}}
\newcommand{\ght}{GH-triple}
\newcommand{\lt}{lazy triple}

\DeclareMathOperator{\ran}{\mathrm{ran}}

\begin{document}

\begin{abstract}
We define an on-line (incremental) algorithm that, given a (possibly
infinite) pseudo-transitive oriented graph, produces a transitive
reorientation. This implies that a theorem of Ghouila-Houri is provable in
$\rca$ and hence is computably true.
\end{abstract}

\maketitle

\section{Introduction}
Let $(V,E)$ be an undirected graph, so that $E$ is a set of unordered pairs
of elements of $V$. We write $a \E b$ to mean that $\{a,b\} \in E$.

An asymmetric and irreflexive relation $\imp$ is an \emph{orientation} of
$(V,E)$ if for every $a,b \in V$ we have $a \E b$ if and only if $a \imp b$
or $b \imp a$. An orientation $\imp$ is \emph{transitive} if for every $a,b,c
\in V$ such that $a \imp b$ and $b \imp c$ we have also $a \imp c$. Graphs
having a transitive orientation are also known as \emph{comparability
graphs}: in fact $E$ is the comparability relation of the strict partial
order $\imp$.

A characterization of comparability graphs was given by Alain Ghouila-Houri
\cite{Ghouila62,Ghouila64} (using a different terminology and dealing only
with finite graphs) and reproved by Paul Gilmore and Alan Hoffman
\cite{gilmore1964}\footnote{\,Further results were obtained by Gallai
\cite{Gallai}.}.

\begin{theorem}\label{thm:comparability}
An undirected graph has a transitive orientation if and only if every cycle
of odd length has a triangular chord.
\end{theorem}

Here a cycle is a sequence of vertices $a_0, \dots, a_k$ such that $a_k =
a_0$ and $a_i \E a_{i+1}$ for every $i<k$. (Notice that a vertex is allowed
to occur more than once in a cycle.) The cycle has odd length if $k$ is odd.
The cycle has a triangular chord if either $a_1 \E a_{k-1}$ or $a_i \E
a_{i+2}$ for some $i<k-1$.

In Figure \ref{figReor:cycles} the left graph has a cycle of length nine with
no triangular chord, while the right one has no cycles of odd length without
triangular chords.

\begin{figure}[h]
\begin{centering}
\begin{tikzpicture}[scale=.6]
\node[buffer] at (3,0.5) (a) {};   
\draw (0.7,-0.04) -- (2,-0.04);
\draw (4,-0.04) -- (5.3,-0.04);
\draw (3,1.7) -- (3,3);         

\foreach \x in {0.7,2,4,5.3}         
{\fill (\x,-0.04) circle (0.8mm);}
\foreach \x in {1.7,3}
{\fill (3,\x) circle (0.8mm);}

\node[buffer] at (15,0.5){};   
\draw (12.7,-0.04) -- (14,-0.04);
\draw (16,-0.04) -- (17.3,-0.04);

\foreach \x in {12.7,14,16,17.3}         
{\fill (\x,-0.04) circle (0.8mm);}
\fill (15,1.7) circle (0.8mm);
\end{tikzpicture}
\caption{\small A graph which is not a comparability graph, to the left, and
a comparability graph, to the right.}\label{figReor:cycles}
\end{centering}
\end{figure}
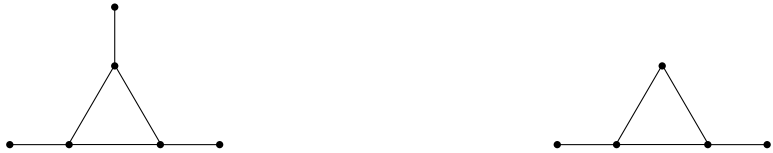

The forward direction of Theorem \ref{thm:comparability} is easily proved.
The backward direction was proved directly by Gilmore and Hoffman, while the
original proof by Ghouila-Houri uses an intermediate step. The latter
approach is also taken in several expositions of the theorem (\cite[Theorem
16.8]{Berge76}, \cite[Theorem 1.7]{fishburn}, \cite[Theorem
11.2.5]{Harzheim}) and hinges on the following notion.

An orientation $\imp$ is \emph{pseudo-transitive} if for every $a,b,c\in V$
such that $a \imp b$ and $b \imp c$ we have also either $a \imp c$ or $c \imp
a$.

Ghouila-Houri proves the backward direction of Theorem
\ref{thm:comparability} by first showing that if every cycle of odd length
has a triangular chord then there exists a pseudo-transitive orientation, and
then that any pseudo-transitive orientation can be further reoriented to
obtain a transitive one.

The effectiveness of Theorem \ref{thm:comparability} has already been
studied, in particular using the framework of reverse mathematics
(\cite{sosoa} is the basic reference in this area), by Jeff Hirst in his PhD
thesis \cite[Theorem 3.20]{hirstPhD}. Hirst indeed showed that a compactness
argument (disguised as an application of Zorn's lemma in \cite{gilmore1964}
and of Rado's theorem in \cite{fishburn,Harzheim}) is necessary for countable
graphs and hence the theorem is not computably true. The following lemma
includes Hirst's theorem and provides a direct proof for it.

\begin{lemma}\label{compgraph}
The following are equivalent over the base system $\rca$:
\begin{enumerate}
\item $\wkl$;
\item every countable graph such that every cycle of odd length has a
    triangular chord has a transitive orientation;
\item every countable graph such that every cycle of odd length has a
    triangular chord has a pseudo-transitive orientation.
\end{enumerate}
\end{lemma}
\begin{proof}
(2) follows from (1) by a straightforward compactness argument, once the
result is proved for finite graphs. The latter can be done in $\rca$,
following any of the proofs mentioned above.

The implication from (2) to (3) is trivial.

To check that (3) implies (1) we use \cite[Lemma IV.4.4]{sosoa} stating that
$\wkl$ is equivalent to the statement that if $f,g\colon \nat \to \nat$ are
injective functions such that $\forall m\, \forall n\, (f(m) \ne g(n))$, then there
exists $X$ such that $\forall m \,(f(m) \in X \land g(m) \notin X)$. Fix $f$
and $g$ as above, and define a graph $(V,E)$ as follows: $V=\{a_n,b_n,c_n,d_n
\mid n\in \nat\}\cup \{x_m, y_m \mid m\in \nat\}$ and $E$ is defined by the
following clauses for each $n$ and $m$:
\[\begin{cases}
c_n \E a_n \E b_n \E c_n \E d_n; &\\
x_m \E a_n & \text{ if } f(m)=n;\\
y_m \E b_n & \text{ if } g(m)=n.
\end{cases}\]
Every cycle of odd length has a triangular chord because every connected
component of $(V,E)$ is isomorphic to a subgraph of the right graph in Figure
\ref{figReor:cycles}. Let $\imp$ be a pseudo-transitive orientation of $E$.
It is easy to check that the set
\[
X=\{n \in \nat \mid b_n \imp a_n \leftarrow c_n \lor b_n \leftarrow a_n \imp c_n \}
\]
contains the range of $f$ and is disjoint from the range of $g$.
\end{proof}

The proof of Lemma \ref{compgraph} yields the following results in the
framework of computability theory and of the Weihrauch lattice (see
\cite{BGP} for an introduction to this research program).

\begin{lemma}\label{incomp}
There exists a computable graph such that every cycle of odd length has a
triangular chord which has no computable pseudo-transitive orientation.

Every computable graph such that every cycle of odd length has a triangular
chord has a low transitive orientation.
\end{lemma}

\begin{lemma}\label{CCan}
Consider the multi-valued functions that map every countable graph such that
every cycle of odd length has a triangular chord to the set of its transitive
(resp.\ pseudo-transitive) orientations. Each of these two multi-valued
functions is Weihrauch equivalent to choice on Cantor space.
\end{lemma}

Starting with \cite{gilmore1964} there has been an interest in algorithms
providing transitive orientations for finite comparability graphs. For
example, the influential textbook \cite{Golumbic} devotes a whole chapter to
algorithmic aspects of comparability graphs, including complexity issues.
However, the first part of Lemma \ref{incomp} shows that there is no
algorithm to (pseudo-)transitively orient countable computability graphs. In
particular, an algorithm which computes a (pseudo-)transitive orientation of
finite comparability graphs cannot work in an incremental way (i.e.\
extending the previous orientation as new vertices are added to the graph),
and thus is not on-line. Here we understand the notion of on-line algorithm
as defined in \cite{BDKM}, which is a recent survey of the theoretical study
of on-line algorithms for computable structures.\medskip

Lemmas \ref{compgraph}, \ref{incomp}, and \ref{CCan} provide an analysis of
the first step in Ghouila-Houri's proof of Theorem \ref{thm:comparability}.
Our main interest is the analysis of the complexity of the second step of
this proof, which is best stated using \emph{oriented graphs}, i.e.\ directed
graphs such that at most one of the edges between two vertices exist. In this
paper we abbreviate \lq oriented graph\rq\ as \emph{ograph}. The notions of
pseudo-transitivity and transitivity are readily extended to ographs, and a
reorientation of an ograph is an ograph obtained by reversing some of the
edges. Then the second step of Ghouila-Houri's proof is the following result.

\begin{theorem}\label{GhouilaHouri}
Every pseudo-transitive ograph has a transitive reorientation.
\end{theorem}

This is the main lemma in \cite{Ghouila62}, the lemma on page 329 in
\cite{Ghouila64}, Theorem 16.7 in \cite{Berge76}, Theorem 1.5 in
\cite{fishburn}, and Theorem 11.2.2 in \cite{Harzheim}. Ghouila-Houri's proof
deals only with finite graphs and uses induction on the number of vertices.
The same proof is presented in \cite{Berge76,fishburn,Harzheim} and extended
to the infinite case by some compactness argument. From this proof it is easy
to extract an algorithm to transitively reorient finite pseudo-transitive
ographs. However, the induction step requires, in a nutshell, partitioning
the set of vertices into two subsets with specific properties, to reorient
each of the induced subographs by induction hypothesis, and then to set the
reorientation between them. Thus this algorithm is not incremental and does
not apply to infinite ographs.

This analysis led us to conjecture that we could obtain results similar to
Lemmas \ref{compgraph}, \ref{incomp} and \ref{CCan} for Theorem
\ref{GhouilaHouri}. We were actually wrong and this is the main result of
this paper. We state this result in various different ways (the first three
items of the theorem correspond to the approaches of Lemmas \ref{compgraph},
\ref{incomp} and \ref{CCan}, respectively).

\begin{maintheorem}\label{main}
\hfill
\begin{enumerate}
  \item $\rca$ proves that every countable pseudo-transitive ograph has a
      transitive reorientation;
  \item every computable pseudo-transitive ograph has a computable
      transitive reorientation;
  \item the multi-valued function that maps a countable pseudo-transitive
      ograph to the set of its transitive reorientations is computable;
  \item there exists an on-line (incremental) algorithm to transitively
      reorient pseudo-transitive ographs;
  \item Player II has a winning strategy for the following game: starting
      from the empty graph, at step $s+1$ player I plays a
      pseudo-transitive extension $(V_s \cup \{x_s\}, \imp_{s+1})$ of the
      pseudo-transitive ograph $(V_s,\imp_s)$ he played at step $s$.
      Player II replies with a transitive reorientation $\prec_{s+1}$ of
      $\imp_{s+1}$ such that $\prec_{s+1}$ extends $\prec_s$ she defined at
      step $s$. Player II wins if and only if she is always able to play.
\end{enumerate}\end{maintheorem}

We concentrate on proving (4) of the Main Theorem, as this easily implies
(1), (2) and (3), while (5) is just a restatement in a different language of
(4) for countable ographs.

We now make precise what we mean by an on-line (incremental) algorithm. We
assume the input to consist of vertices coming one at a time together with
all information about the edges connecting them to previous vertices. (So at
step $s$ the size of the input increases of at most $s$.) When the algorithm
sees a new vertex, it must reorient all the edges connecting it to previous
vertices while preserving the reorientations already set at previous stages.

We deal explicitly only with countable ographs; however it is easily seen
that our algorithm applies to ographs of any cardinality, as long as the set
of vertices can be well-ordered.

An upper bound for the complexity of the algorithm we define (when applied to
finite pseudo-transitive ographs) is $O(|V|^3)$. The problem of orienting
comparability graphs can be solved by an algorithm with complexity $O(\delta
\cdot |E|)$, where $\delta$ is the maximum degree of a vertex (\cite[Theorem
5.33]{Golumbic}), and further fine-tuning has been subsequently
made.\smallskip

We now describe the organization of the paper. Section \ref{secReor:prelim}
contains the preliminary definitions and a presentation of two
pseudo-transitive ographs with transitive reorientations which are the main
obstacles in designing the algorithm. Sections \ref{secReor:first} and
\ref{secReor:second} analyze in detail these two configurations. In Section
\ref{sec:algo} we present the on-line algorithm and prove its correctness. We
also sketch the ideas needed to obtain the upper bound for the complexity
mentioned above.

\section{Preliminaries}\label{secReor:prelim}

In the introduction we have already introduced our terminology and we now
give the formal definitions of the central notions.

\begin{definition}
An ograph $(V,\imp)$ is \emph{transitive} if for each $a,b,c\in V$, if $a\imp
b\imp c$, then $a\imp c$. $(V,\imp)$ is \emph{pseudo-transitive} if for each
$a,b,c\in V$, if $a\imp b\imp c$, then $a\imp c$ or $c\imp a$.

A relation $R$ on $V$ is a \emph{reorientation} of $\imp$, if for each
$a,b\in V$, if $a\imp b$ then either $a \R b$ or $b \R a$ and if $a \R b$
then either $a\imp b$ or $b\imp a$.

A \emph{transitive reorientation} of $(V, \imp)$ is a reorientation of $(V,
\imp)$ which is also transitive. In this case we often use $\prec$ in place
of $R$.

A triple $(V, \imp, \prec)$ is a \emph{Ghouila-Houri triple} (\emph{\ght} for
short) if $(V, \imp)$ is a pseudo-transitive ograph and $\prec$ a transitive
reorientation of $\imp$.
\end{definition}

Notice that each reorientation $R$ of $(V, \imp)$ preserves both
$\imp$-comparability and $\imp$-incomparability. In other words, the
undirected graphs associated with $(V,\imp)$ and with $(V, R)$ coincide.

\begin{notation} Let $(V,\imp)$ be an ograph and $a,b,c\in V$.
\begin{itemize}
\item $a-b$ means that either $a\imp b$ or $b\imp a$;
\item $N(a)= \{b \in V \mid a-b\}$ is the \emph{neighborhood of $a$};
\item $a\mid b$ means that neither $a \imp b$ nor $b \imp a$;
\item when we write \lq $a-b$ by \pt c\rq\ we mean that we know that $\imp$
    is pseudo-transitive and we are deducing $a-b$ because we have either
    $a \imp c \imp b$ or $b \imp c \imp a$.
\end{itemize}
\end{notation}

\begin{definition}
Let $(V,\imp)$ be an ograph. If $V' \supseteq V$ we say that $(V', \imp')$ is
an \emph{extension} of $(V,\imp)$ if $(V', \imp')$ is an ograph such that for
every $a,b \in V$ we have $a \imp b$ if and only if $a \imp' b$.
\end{definition}

An on-line algorithm computing a transitive reorientation of a
pseudo-transitive ograph must produce at each step a reorientation which can
further be extended, in the sense made precise by the following definition.

\begin{definition}
A \ght\ $(V,\imp,\prec)$ is \emph{extendible} if for every $(V \cup \{x\},
\imp')$, pseudo-transitive extension of $(V,\imp)$, there exists $\prec'$
which extends $\prec$ and is such that $(V \cup \{x\},\imp',\prec')$ is a
\ght.
\end{definition}

Some simple cases of \ght s which are not extendible are depicted in Figures
\ref{fig:trans} and \ref{fig:2+2}.
\begin{ex}\label{ex:trans}
In Figure \ref{fig:trans} we have the \emph{transitive triangle examples}:
$\imp$ is transitive on $\{a,b,c\}$ and the transitive reorientation is
defined by $a \prec c \prec b$. Notice that in the left ograph we have $a
\imp c \leftarrow b$, while in the right one we have $a \leftarrow c \imp b$:
in both cases all edges involving the vertex $c$ have the same direction. We
can add a vertex $x$ connected to $c$ by an edge going in the same direction
and connected with neither $a$ nor $b$. Then $(\{a,b,c,x\}, \imp')$ is
pseudo-transitive and if $\prec'$ is a reorientation of $\imp'$ extending
$\prec$ we must have either $x \prec' c$ or $c \prec' x$: both choices lead
to the failure of transitivity of $\prec'$.
\end{ex}

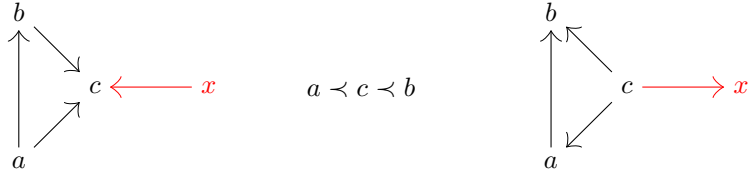
\begin{figure}
\begin{centering}
\begin{tikzpicture}
\node (a) at (0,0){$a$};
\node (b) at (0,2){$b$};
\node (c) at (1,1){$c$};
\node[red] (x) at (2.5,1){$x$};

\draw[-{Classical TikZ Rightarrow[length=1.5mm]}]  (a)--(b);
\draw[-{Classical TikZ Rightarrow[length=1.5mm]}] (b)--(c);
\draw[-{Classical TikZ Rightarrow[length=1.5mm]}] (a)--(c);
\draw[-{Classical TikZ Rightarrow[length=1.5mm]},red] (x)--(c);

\node at (4.5,1){$a \prec c \prec b$};

\node (a0) at (7,0){$a$};
\node (b0) at (7,2){$b$};
\node (c0) at (8,1){$c$};
\node[red] (x0) at (9.5,1){$x$};

\draw[-{Classical TikZ Rightarrow[length=1.5mm]}] (a0)--(b0);
\draw[-{Classical TikZ Rightarrow[length=1.5mm]}] (c0)--(b0);
\draw[-{Classical TikZ Rightarrow[length=1.5mm]}] (c0)--(a0);
\draw[-{Classical TikZ Rightarrow[length=1.5mm]},red] (c0)--(x0);
\end{tikzpicture}
\caption{\small The transitive triangle examples.}\label{fig:trans}
\end{centering}
\end{figure}

\begin{ex}\label{ex:2+2}
In Figure \ref{fig:2+2} we have the \emph{$2 \oplus 2$ example}: there are
two edges $a \imp c$ and $b \imp d$ (with no other edges between these four
vertices) and the transitive reorientation defined by $a \prec c$ and $d
\prec b$. In the left ograph we add a vertex $x$ such that $a \imp' x$, $b
\imp' x$, $x \mid 'c$ and $x \mid' d$. Then $(\{a,b,c,d,x\}, \imp')$ is
pseudo-transitive. Suppose $\prec'$ were a transitive reorientation of
$\imp'$ extending $\prec$: since $a-'x$ and $x \mid 'c$, then $a \prec c$
implies $a \prec' x$; since $b-'x$ and $x \mid 'd$, then $d \prec b$ implies
$x \prec' b$. But $a \prec' x \prec' b$ is not compatible with $a\mid b$. The
situation in the right ograph is the same as the previous one as far as the
first four vertices are concerned, but the new vertex $x$ is now such that $x
\imp' c$, $x \imp' d$, $x\mid 'a$ and $x\mid 'b$. We can argue analogously to
show that $(\{a,b,c,d,x\}, \imp')$ is a pseudo-transitive ograph with no
transitive reorientation extending $\prec$.
\end{ex}

\begin{figure}
\begin{centering}
\begin{tikzpicture}
\node (a2) at (0,0){$a$};
\node (b2) at (2.5,2){$b$};
\node (c2) at (0,2){$c$};
\node (d2) at (2.5,0){$d$};

\node[red] (x2) at (1.25,1){$x$};

\draw[-{Classical TikZ Rightarrow[length=1.5mm]}] (a2)--(c2);
\draw[-{Classical TikZ Rightarrow[length=1.5mm]}] (b2)--(d2);
\draw[-{Classical TikZ Rightarrow[length=1.5mm]},red] (a2)--(x2);
\draw[-{Classical TikZ Rightarrow[length=1.5mm]},red] (b2)--(x2);

\node (a3) at (7,0){$a$};
\node (b3) at (9.5,2){$b$};
\node (c3) at (7,2){$c$};
\node (d3) at (9.5,0){$d$};

\node[red] (x3) at (8.25,1){$x$};

\draw[-{Classical TikZ Rightarrow[length=1.5mm]}] (a3)--(c3);
\draw[-{Classical TikZ Rightarrow[length=1.5mm]}] (b3)--(d3);
\draw[-{Classical TikZ Rightarrow[length=1.5mm]},red] (x3)--(c3);
\draw[-{Classical TikZ Rightarrow[length=1.5mm]},red] (x3)--(d3);

\node at (4.7,1.5){$a \prec c$};
\node at (4.7,0.5){$d \prec b$};
\end{tikzpicture}
\caption{\small The $2 \oplus 2$ example.}\label{fig:2+2}
\end{centering}
\end{figure}
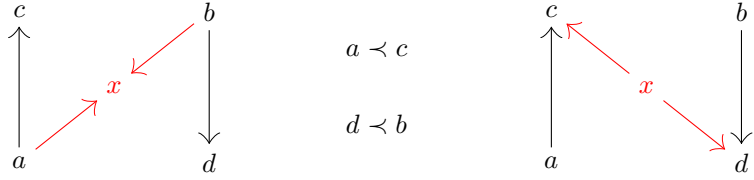

We eventually show that the examples above are the only obstructions to
extendibility of a \ght. To do this we analyze in detail Examples
\ref{ex:trans} and \ref{ex:2+2} using the following notions.

\begin{definition}\label{def:N(X)}
Let $(V,\imp,\prec)$ be a \ght. If $(V\cup\{x\}, \imp')$ is a pseudo-transitive
extension of $(V,\imp)$ define
\begin{align*}
  N^+(x)= & \{a\in N(x) \mid \forall b\, (a \prec b \Imp b\in N(x))\}; \\
  N^-(x)= & \{a\in N(x) \mid \forall b\, (b \prec a \Imp b\in N(x))\}.
\end{align*}
(Here $N(x)$ is the neighborhood of $x$ in $(V\cup\{x\}, \imp')$.)
\end{definition}

\begin{remark}
Under the hypothesis of the previous definition we have that if $a \in N(x)
\setminus N^+(x)$, $b \in N^+(x)$ and $a-b$, then $a \prec b$. In fact, since
$a \notin N^+(x)$ there is $d \succ a$ with $d \mid x$. If $b \prec a$, then
$b \prec d$ against $b \in N^+(x)$. Thus $a\prec b$.

Similarly, if $c \in N(x) \setminus N^-(x)$, $b \in N^-(x)$, and $b-c$, then
$b \prec c$.
\end{remark}

The next lemma states some properties of extendible \ght s.

If $A,B \subseteq V$ we write $A \prec B$ to mean that $a \prec b$ for every
$a \in A$ and $b \in B$.

\begin{lemma}\label{Nec}
Let $(V,\imp,\prec)$ be an extendible \ght. Then for any $(V\cup\{x\},
\imp')$ pseudo-transitive extension of $(V,\imp)$ we have:
\begin{enumerate}
	\item \label{NecUno} $N(x)=N^+(x)\cup N^-(x)$;
	\item \label{NecDue} $N^-(x)\setminus N^+(x) \prec N^+(x)\setminus N^-(x)$.
\end{enumerate} 	
\end{lemma}
\begin{proof}
If condition (1) does not hold for some pseudo-transitive extension
$(V\cup\{x\}, \imp')$, then there exist $c\in N(x)$ and $a,b\not\in N(x)$
such that $a \prec c \prec b$. This impedes both $x \prec' c$ and $c \prec'
x$ for any transitive reorientation of $\imp'$ with ${\prec'} \supseteq
{\prec}$. (Notice that we found in $(V,\imp,\prec)$ a copy of one of the
transitive triangle examples.)

If condition (2) does not hold for some pseudo-transitive extension
$(V\cup\{x\}, \imp')$, then there exist $a\in N^-(x)\setminus N^+(x)$ and
$b\in N^+(x)\setminus N^-(x)$ such that $a\nprec b$. Since $a\in
N^-(x)\setminus N^+(x)$ there exists $c$ such that $a \prec c$ and $c\mid
'x$. Since $b\in N^+(x)\setminus N^-(x)$, there exists $d$ such that $d \prec
b$ and $d\mid 'x$. If $\prec'$ were a transitive reorientation of $\imp'$
with ${\prec'} \supseteq {\prec}$ then these conditions imply respectively $a
\prec' x$ and $x \prec' b$; it would follow $a \prec' b$, contrary to $a
\nprec b$. (Notice that in this case we found in $(V,\imp,\prec)$ a copy of
the $2 \oplus 2$ example.)
\end{proof}

\section{Avoiding the transitive triangle examples}\label{secReor:first}

This section is devoted to a careful study of the first condition of Lemma
\ref{Nec}. The next lemma shows that this condition captures precisely the
lack of the transitive triangle examples. Recall that in that situation
$(V,\imp, \prec)$ is a \ght. Moreover, there exist $a,b,c\in V$ such that
$a\prec c\prec b$ and the new vertex $x$ is connected with $c$, but not with
$a$ and $b$. Notice that this might happen only if $a$, $b$, $c$ form a
transitive triangle and either $a \imp c \leftarrow b$ or $a \leftarrow c
\imp b$.

\begin{lemma}\label{CharProp1}
Let $(V,\imp,\prec)$ be a \ght\ and $(V \cup \{x\}, \imp')$ be a pseudo-transitive
extension of $(V,\imp)$, then $N(x)=N^+(x) \cup N^-(x)$ is equivalent to
$\forall a,b,c\in V\, (a \prec c \prec b \land x-'c \Imp x-'a \lor x-'b)$.
\end{lemma}
\begin{proof}
Notice that $c \in N(x) \setminus N^+(x)\cup N^-(x)$ means that there exist
$a$ and $b$ such that $a \prec c \prec b$ and $a,b \notin N(x)$. From this
observation the equivalence is immediate.
\end{proof}

Lemma \ref{CharProp1} involves all possible pseudo-transitive extensions of
$(V,\imp)$ by one vertex. It is convenient to have a characterization of the
\ght s such that $N(x) = N^+(x)\cup N^-(x)$ for every pseudo-transitive
extension, which involves only the \ght\ itself. To this end we introduce two
formulas, $\Phi$ and $\Psi$. In order to do this, we define formulas
$\varphi(a,b,c)$ and $\psi(a,b,c)$ which do not mention the reorientation
$\prec$. Notice that Lemmas \ref{Nec} and \ref{CharProp1} imply that the non
extendibility of $\prec$ may be caused by only three vertices. With this in
mind, it is not hard to understand the rationale for $\varphi(a,b,c)$,
$\psi(a,b,c)$, $\Phi$, and~$\Psi$.

\begin{definition}\label{def:Varphi}
Let $(V,\imp,\prec)$ be a \ght.
Let $\varphi(a,b,c)$ assert the existence of $e_0, \dots, e_n \in V$ such that:
\begin{enumerate}[\quad ($\varphi_1$)]
  \item $c\imp e_0$;
  \item $\forall i < n \,((a\imp e_i \land b\imp e_i \imp e_{i+1}) \lor
      (e_{i+1} \imp e_i \imp b \land e_i\imp a))$;
  \item $a \imp e_n \imp b \lor b \imp e_n \imp a$.
\end{enumerate}
Then $\Phi$ is
\[
\forall a,b,c \in V \, (a\imp c\leftarrow b \land a \prec c \prec b \Imp \varphi(a,b,c)).
\]

Symmetrically, let $\psi(a,b,c)$ assert the existence of $e_0, \dots, e_n \in V$
such that:
\begin{enumerate}[\quad ($\psi_1$)]
  \item $e_0 \imp c$;
  \item $\forall i < n \, ((a\imp e_i \land b\imp e_i \imp e_{i+1}) \lor
      (e_{i+1} \imp e_i \imp b \land e_i\imp a))$;
  \item $a \imp e_n \imp b \lor b \imp e_n \imp a$.
\end{enumerate}
Then $\Psi$ is
\[
\forall a,b,c \in V \, (a \leftarrow c \imp b \land a \prec c \prec b \Imp \psi(a,b,c)).
\]
\end{definition}

Notice that the only difference between $\varphi$ and $\psi$ occurs in
conditions ($\varphi_1$) and ($\psi_1$), where the direction of the edge is reversed. $\Phi$ and $\Psi$
further differ in applying to triples such that $a\imp c\leftarrow b$ and $a
\leftarrow c \imp b$ respectively.

\begin{remark}\label{varphi:scambio}
Let $(V,\imp,\prec)$ be a \ght. Fix $a,b,c \in V$. If $e_0, \dots, e_n$ witness
$\varphi(a,b,c)$ (or $\psi(a,b,c)$) then they witness $\varphi(b,a,c)$ (resp.
$\psi(b,a,c)$) as well.
\end{remark}

The following duality principle is useful to avoid checking $\Phi$ and $\Psi$
separately.

\begin{remark}\label{dualityPhi}
Using Remark \ref{varphi:scambio} it follows immediately that $(V, \imp,
\prec)$ satisfies $\Phi$ if and only if $(V, \leftarrow, \succ)$ (i.e.\ the
ograph and the reorientation where all edges are reversed) satisfies $\Psi$.
\end{remark}

We start with some properties concerning basic facts about $\varphi$ and
$\psi$.

\begin{property}\label{ei-x}
Let $(V,\imp)$ be a pseudo-transitive ograph. Suppose that $a \imp c
\leftarrow b$ and $\varphi(a,b,c)$ is witnessed by $e_0, \dots, e_n$. Then
there exists $k \le n$ such that $e_k, \dots, e_n$ witness $\varphi(a,b,d)$
for each $d\in V$ such that $d\mid c$ and $a-d-b$.

The same holds starting from $a\leftarrow c \imp b$ and $\psi(a,b,c)$, and
concluding that $e_k, \dots, e_n$ witness $\psi(a,b,d)$.
\end{property}
\begin{proof}
Suppose we are in the first case, i.e.\ $a \imp c \leftarrow b$ and $e_0,
\dots, e_n$ witness $\varphi(a,b,c)$. Let $k \le n$ be largest such that $c
\imp e_k$, and notice that $e_k, \dots, e_n$ witness $\varphi(a,b,c)$ as
well.

We claim that $e_i \imp c$ for all $i$ such that $k<i \le n$. The claim is
proved by a \lq backward\rq\ induction. We obtain $e_n-c$ by ($\varphi_3$)
and \pt b or \pt a. Hence $e_n \imp c$ by our assumption (unless $n=k$).
Suppose now that $e_{i+1} \imp c$. If $e_i \imp a$, then $e_i - c$ by \pt a.
Otherwise, $e_i \imp e_{i+1}$ by ($\varphi_2$) and so $e_i - c$ by
\pt{e_{i+1}}. Hence, if $i>k$ we have $e_i \imp c$.

Let now $d$ be such that $d\mid c$ and $a-d-b$. In particular we have $a \imp
d \leftarrow b$. Notice that to check that $e_k, \dots, e_n$ witness
$\varphi(a,b,d)$ conditions ($\varphi_2$) and ($\varphi_3$) are identical to
conditions ($\varphi_2$) and ($\varphi_3$) of $\varphi(a,b,c)$, since they
concern only the vertices $a$ and $b$. We are left to prove that condition
($\varphi_1$) is satisfied, namely that $d \imp e_k$. Since $d\mid c$ and $c
\imp e_k$ it suffices to show that $d - e_k$.

To this end we prove that indeed we have $e_i-d$ for all $i$ such that $k \le
i \le n$, again by a \lq backward\rq\ induction. Since $a \imp d \leftarrow
b$ and either $e_n \imp a$ or $e_n \imp b$ by ($\varphi_3$), we have $e_n-d$
by either \pt a or \pt b. Now, assuming $i \ge k$ and $e_{i+1}-d$ so that
$d-e_{i+1}-c$, we must have $e_{i+1} \imp d$ because $e_{i+1} \imp c$ by the
choice of $k$. If $a \imp e_i$ condition ($\varphi_2$) of $\varphi(a,b,c)$
implies $e_i\imp e_{i+1}$ and hence $e_i-d$ by \pt{e_{i+1}}. If $e_i\imp a$,
then $e_i-d$ by \pt a, since $a \imp d$.

If $a\leftarrow c \imp b$ and $e_0, \dots, e_n$ witness $\psi(a,b,c)$ the
argument is similar with obvious changes.
\end{proof}

\begin{property} \label{Incom}
Let $(V,\imp)$ be a pseudo-transitive ograph and let $v,u,e_0, \dots,  e_n
\in V$. Suppose $u \mid v$, $u-e_0$ and $\forall i< n \, (v\imp e_i \imp
e_{i+1} \lor e_{i+1} \imp e_i \imp v)$. Then $u-e_i$ for each $i \le n$.
\end{property}
\begin{proof}
The proof is by induction on $i$. The base case holds by assumption, so
assume $u-e_i$ for $i<n$. If $u \imp e_i$, then $v\imp e_i$ because $u \mid
v$. Thus $e_i \imp e_{i+1}$ and $u-e_{i+1}$ by \pt{e_i}. If $e_i \imp u$ the
argument is symmetric inverting the arrows.
\end{proof}

%
%

We can now show that $\Phi$ and $\Psi$ are sufficient for the first condition
of Lemma \ref{Nec}.

\begin{lemma}\label{characterization}
Let $(V,\imp, \prec)$ be a \ght. If $\Phi$ and $\Psi$ are satisfied, then for each
$(V\cup\{x\}, \imp')$ pseudo-transitive extension of $(V,\imp)$ we have
$N(x)=N^+(x)\cup N^-(x)$.
\end{lemma}
\begin{proof}
Fix $(V\cup\{x\}, \imp')$. By Lemma \ref{CharProp1} it suffices to show that
for any $a, b, c \in V$ such that $a \prec c \prec b$ and $x-'c$ either
$x-'a$ or $x-'b$.

If $b \imp c \imp a$ then $x \imp' c$ implies $x-'a$, while $c \imp' x$
implies $x-'b$. If $a \imp c \imp b$ the situation is similar.

If $a \imp c \leftarrow b$ then $\Phi$ implies that $\phi(a,b,c)$ holds. Let
$e_0, \dots, e_n$ witness $\varphi(a,b,c)$. Assume $x-'c$. If $c \imp' x$,
then both $a-'x$ and $b-'x$ follow immediately by \pt c. Otherwise we have $x
\imp' c$, and suppose towards a contradiction that $x\mid 'a$ and $x\mid 'b$.
Notice that $x \imp' c \imp e_0$ implies $x-'e_0$. Hence by condition ($\varphi_2$) and
Property \ref{Incom} it holds that $\forall i\le n \,(x-'e_i)$. In particular
we have $x-'e_n$, and then one of $x-'a$ and $x-'b$ by $\pt{e_n}$ follows by
($\varphi_3$).

If $a \leftarrow c \imp b$ we argue similarly, using $\Psi$.
\end{proof}

We now prove that $\Phi$ and $\Psi$ are necessary conditions for
$N(x)=N^+(x)\cup N^-(x)$.

\begin{lemma} \label{sufTrian}
Let $(V,\imp, \prec)$ be a \ght\ such that one of $\Phi$ and $\Psi$ fails. Then there is a
pseudo-transitive extension $(V\cup \{x\},\imp')$ of $(V,\imp)$ such that
$N(x)\ne N^+(x)\cup N^-(x)$ and hence $(V,\imp, \prec)$ is not extendible by Lemma
\ref{Nec}.
\end{lemma}
\begin{proof}
We assume the failure of $\Phi$: if $\Psi$ fails the argument is symmetric.

Let $a,b,c \in V$ be such that $a \imp c \leftarrow b$, $a \prec c \prec b $
and $\neg\varphi(a,b,c)$. We fix $x \notin V$ and define an extension $\imp'$
of $(V,\imp)$ to $V \cup \{x\}$ in stages, as an increasing union ${\imp'} =
{\bigcup_{n\in \nat} \imp_n}$. For each stage $n$, $\imp_n$ is defined as
follows:
\begin{itemize}
  \item $\imp_0$ extends $\imp$ by adding the single edge $x \imp_0 c$;
  \item $\imp_{n+1}$ extends $\imp_n$ by adding edges
\[\begin{cases}
x \imp_{n+1} u & \text{ if } \exists v ((x\imp_n v \imp u ) \lor (u \imp v \imp_n x))
 \text{ and } a \imp u \leftarrow b; \\
u \imp_{n+1} x & \text{ if } \exists v ((x\imp_n v \imp u ) \lor (u \imp v \imp_n x))
 \text{ and } a \leftarrow u \imp b. \\
\end{cases}\]
\end{itemize}
Notice that $x-'c$ but $x\mid 'a$ and $x\mid 'b$ and hence $c \in N(x)$ but
$c \notin N^+(x)\cup N^-(x)$. Therefore to complete the proof it suffices to
check the pseudo-transitivity of $\imp'$. We first make a couple of
preliminary observations.

\begin{claim}\label{Fact1}
For all $u \in V$ such that there exists $v \in V$ satisfying either $x \imp'
v \imp u$ or $u \imp v \imp' x$ we have $a-u-b$.
\end{claim}
\begin{prf}
Let us first suppose that $x \imp' v \imp u$ holds. By definition of $\imp'$
(or by hypothesis when $v=c$) it holds that $a\imp v\leftarrow b$. Hence
$a-u-b$ by \pt v. If $u \imp v \imp' x$ the argument is similar.
\end{prf}

\begin{claim}\label{Fact2}
If $u \in V$ is such that $u \neq c$ and $u -_1 x$ then $c \imp u$.
\end{claim}
\begin{prf}
Let us suppose that $u\neq c$ and $u -_1 x$, so that $u -_0 x$ does not hold.
The definition of $\imp_1$ implies that for some $v$ we have either $x\imp_0
v \imp u$ or $u \imp v \imp_0 x$. Since the only $v$ such that $v -_0 x$ is
$c$ and $x \imp_0 c$ we must have the first possibility with $v=c$, so that
$c \imp u$ holds.
\end{prf}

In order to show that $\imp'$ is pseudo-transitive, we have to consider the
following three cases for $v,u \in V$:
\begin{enumerate}[a.]
\item $v\imp' x\imp' u$. Then $v-u$ because $v\imp a \imp u$ by definition
    of  $\imp'$;
\item $x\imp' v \imp u$. Then Claim \ref{Fact1} guarantees that $a-u-b$.
    Let $n$ be the least stage such that $x \imp_n v$. If $a \imp u
    \leftarrow b$ or $a\leftarrow u \imp b$, then $x-_{n+1} u$ by
    definition of $\imp_{n+1}$. Thus we assume that either $a\imp u\imp b$
    or $b\imp u\imp a$. Since $n$ is the minimum stage such that $x\imp_n
    v$, there exists $e_{n-2}$ such that $x-_{n-1} e_{n-2}-v$ and
    $x\imp_{n-1} e_{n-2} \Biimp e_{n-2}\imp v$. Notice that $x -_{n-2}
    e_{n-2}$ does not hold, otherwise we would have $x \imp_{n-1} v$.
    Analogously, there must be an $e_{n-3}$ such that $x-_{n-2} e_{n-3}-
    e_{n-2}$ and $x\imp_{n-2} e_{n-3} \Biimp e_{n-3}\imp e_{n-2}$. For each
    step $i<n$, we can repeat this search of $e_{i-2}$ witnessing that
    $x-_i e_{i-1}$. After $n-1$ steps we get to $x-_1 e_0$ and, since $x
    -_0 e_0$ does not hold, $e_0 \neq c$. This means, by Claim \ref{Fact2},
    that $c\imp e_0$. Let $j$ be maximum such that $c \imp e_j$ and set
    $e_{n-1}=v$ and $e_n=u$. We claim that $e_j, \dots, e_n$ witness
    $\varphi(a,b,c)$. To this end we need to check the three clauses in the
    definition of $\varphi(a,b,c)$:
\begin{enumerate}[($\varphi_1$)]
\item $c\imp e_j$ by hypothesis.
\item Fix $i<n$: $e_i-e_{i+1}$ holds by our choice of the sequence of the
    $e_i$'s and we have either $a\imp e_i \leftarrow b$ or $a \leftarrow
    e_i \imp b$ by definition of $\imp_i$. Moreover, if $x \imp_{i+1}
    e_i$, then $b \imp e_i$, by definition of $\imp_{i+1}$, and $e_i \imp
    e_{i+1}$, by choice of $e_i$. If $e_i \imp_{i+1} x$ the argument is
    similar.
\item $a\imp e_n \imp b$ or $b\imp e_n \imp a$ by hypothesis.
\end{enumerate}
  \item $u \imp v \imp' x$. This is similar to the previous case.\qedhere
\end{enumerate}
\end{proof}

Summarizing the results obtained in Lemma \ref{characterization} and Lemma
\ref{sufTrian} we obtain:

\begin{corollary}\label{N(x)}
Let $(V,\imp, \prec)$ be a \ght. The following are equivalent:
\begin{enumerate}
\item for each pseudo-transitive extension $(V\cup \{x\}, \imp')$ of $(V,\imp)$,
it holds that $ N(x)=N^+(x)\cup N^-(x)$;
\item  $\Phi$ and $\Psi$ are satisfied.
\end{enumerate}
\end{corollary}

\section{Avoiding the $2 \oplus 2$ example}\label{secReor:second}

This section is devoted to study more carefully the second condition of Lemma
\ref{Nec}. The next lemma shows that, assuming that $N(x)=N^+(x)\cup N^-(x)$,
this condition captures precisely the lack of the $2 \oplus 2$ example.
Recall that in that example $(V,\imp, \prec)$ is a \ght\ and there exist
$a,b,c,d\in V$ such that $a\prec c$, $d \prec b$, $a\mid b$, $a\mid d$,
$c\mid b$, and $c\mid d$. Then, a new vertex $x$ is connected with $a$ and
$b$ but not with $c$ and $d$, or vice versa. Notice that this is possible
only if either $a \imp c$ and $b \imp d$, or $c \imp a$ and $d \imp b$.

\begin{lemma}\label{CharProp2}
Let $(V,\imp, \prec)$ be a \ght\ and $(V\cup\{x\}, \imp')$ a pseudo-transitive extension of
$(V,\imp)$. We use $\Lambda$ to denote the following property of
$(V\cup\{x\}, \imp')$ and $\prec$:
\[
\forall a,b,c,d\in V \,(a\mid b \land c\mid d \land a \prec c \land d \prec b \land x-'a \land x-'b
\Imp x-'d \lor x-'c)
\]
Then:
\begin{enumerate}[(1)]
\item if $\Lambda$ holds then $N^-(x)\setminus N^+(x) \prec N^+(x)
    \setminus N^-(x)$;
\item\label{CharProp22} if $N(x)=N^+(x)\cup N^-(x)$ and $N^-(x) \setminus N^+(x) \prec N^+(x)
    \setminus N^-(x)$ then $\Lambda$ holds.
\end{enumerate}
\end{lemma}
\begin{proof}
$(1)$ Assume $N^-(x)\setminus N^+(x) \nprec N^+(x)\setminus N^-(x)$, i.e.\
there exist $a\in N^-(x)\setminus N^+(x)$ and $b\in N^+(x)\setminus N^-(x)$
such that $a\nprec b$. Since $a \notin N^+(x)$ there is $c \succ a$ with
$c\mid x$. Since $b \notin N^-(x)$ there is $d \prec b $ with $d\mid x$. If
$b \prec a$, then $b \prec c$ but this is impossible since $c\mid x$ and
$b\in N^+(x)$. Since we are assuming $a\nprec b$ we have $a\mid b$.

We claim that $c\mid d$ also holds. Since $a-x-b$, but $a\mid b$, then either
$a \imp x \leftarrow b$ or $a \leftarrow x \imp b$. The argument for the two
cases is similar, so let us assume that $a \imp x \leftarrow b$. This implies
$a\imp c$ and $b\imp d$ because $c\mid x$ and $x\mid d$. Hence if $c\imp d$,
then $a-d$ by \pt c. Since $a\mid b$ and $d \prec b$, it must be $d \prec a$
but this contradicts $a\in N^-(x)$. If $d\imp c$, then $c-b$ by \pt d. Since
$a\mid b$ and $a \prec c$, it must be $b \prec c$  which contradicts $b\in
N^+(x)$. We have thus shown that $c\mid d$ as claimed.

Now $a$, $b$, $c$ and $d$ witness the failure of $\Lambda$.\smallskip

$(2)$ Assume that $a,b,c,d\in V$ witness the failure of $\Lambda$. Then $a
\notin N^+(x)$, $b \notin N^-(x)$ and $a\nprec b$. If $N(x)=N^+(x)\cup
N^-(x)$ holds then $a \in N^-(x)$ and $b \in N^+(x)$, showing that $N^-(x)
\setminus N^+(x) \prec N^+(x) \setminus N^-(x)$ fails.
\end{proof}

\begin{obs}\label{c mid b}
Notice that the first four conjuncts of the antecedent of the implication
appearing in $\Lambda$ imply that $a$, $b$, $c$ and $d$ form a $2 \oplus 2$
because $c\mid b$ and $a\mid d$ follow from these. In fact, if $c-b$, then $a
\prec c$ and $a\mid b$ imply that $b \prec c$, but then $d-c$ contrary to the
assumption. A similar argument shows that $a\mid d$.
\end{obs}

We now define two formulas $\Theta$ and $\Sigma$ characterizing the
reorientations such that the condition $\Lambda$ of Lemma \ref{CharProp2} is
satisfied whenever $N(x)=N^+(x)\cup N^-(x)$. As for $\Phi$ and $\Psi$, the
main feature of $\Theta$ and $\Sigma$ is that they mention only $(V,\imp)$
and $\prec$. In order to define $\Theta$ and $\Sigma$ it is necessary to
define $\theta(a,b,c,d)$ and $\sigma(a,b,c,d)$ (which do not mention
$\prec$).

\begin{definition}\label{def:Theta}
Let $(V,\imp, \prec)$ be a \ght. Let $\theta(a,b,c,d)$ assert the existence of $e_0, \dots, e_n \in V$
such that:
\begin{enumerate}[\quad($\theta_1$)]
\item $e_0 \imp b$;
\item $\forall i< n \,(e_{i+1}\imp e_i \imp d)$;
\item $d \imp e_n$;
\item $e_n\mid a$.
\end{enumerate}\smallskip
Then $\Theta$ is
\[
\forall a,b,c,d \in V \, (a\imp c \land b\imp d \land a\mid b \land c \mid d \land a \prec c \land d \prec b \Imp \theta(a,b,c,d) \lor \theta(b,a,d,c)).
\]
Symmetrically, let $\sigma(a,b,c,d)$ assert the existence of $e_0, \dots, e_n \in V$ such that:
\begin{enumerate}[\quad($\sigma_1$)]
\item $d \imp e_0$;
\item $\forall i< n \, (b \imp e_i \imp e_{i+1})$;
\item $e_n \imp b$;
\item $e_n\mid c$.
\end{enumerate}\smallskip
Then $\Sigma$ is
\[
\forall a,b,c,d \in V \, (a\imp c \land b\imp d \land a\mid b \land c \mid d \land a \prec c \land d \prec b \Imp \sigma(a,b,c,d) \lor \sigma(b,a,d,c)).
\]
\end{definition}

\begin{ex}
Suppose $(\{a,b,c,d,e\}, \imp)$ is the pseudo-transitive graph whose only
edges are $a\imp c$, $b\imp d$ and $d \imp e \imp b$. Then $\theta(a,b,c,d)$
and $\sigma(a,b,c,d)$ hold with $n=0$ and $e_0=e$. Thus a $2 \oplus 2$ such
as the one obtained restricting $\imp$ to $\{a,b,c,d\}$ can satisfy $\theta$
and $\sigma$ simply because one of its edges belongs to a non transitive
triangle. See the first paragraph of the proof of lemma
\ref{PsiPhiThetaSigma} below for more on this.
\end{ex}

\begin{remark}\label{rem:thetaimp}
Let $(V,\imp, \prec)$ be a \ght. Suppose $e_0, \dots, e_n$ witness $\theta(a,b,c,d)$ for some $a,b,c,d
\in V$. Clearly, if there is an $i> 0$ such that $e_i \imp b$, then $e_i,
\dots, e_n$  witness $\theta(a,b,c,d)$ as well. Thus we can assume that for
every $i\le n$ with $i>0$ we have $b \imp e_i$ whenever $b-e_i$. Under this
assumption it actually holds that $b \imp e_i$ holds for every $i\le n$ with
$i>0$. In fact, $b-e_n$ by \pt d and if $b \imp e_{i+1}$, then $b-e_i$ by
\pt{e_{i+1}}.
\end{remark}

Before proving the usefulness of $\Theta$ and $\Sigma$, we would like to
comment on their mutual relationship and on the difference between the
connection between $\Theta$ and $\Sigma$ and the connection between $\Phi$
and $\Psi$. Let $(V,\imp, \prec)$ be a \ght\ and suppose $a,b,c,d \in V$
satisfy the antecedent of $\Theta$ and $\Sigma$ (which is the same). Consider
a pseudo-transitive extension $(V\cup\{x,y\}, \imp')$ such that $a\imp' x
\leftarrow' b$ and $c\leftarrow' y \imp' d$. The two extensions correspond
respectively to the left and right ograph of Figure \ref{fig:trans}. As
explained at the beginning of this section, if either $x$ is incomparable
with both $c$ and $d$ or if $y$ is incomparable with both $a$ and $b$, then
$(V,\imp, \prec)$ is not extendible. We emphasize that under these hypotheses we could
have both $x$ and $y$ witnessing the non extendibility of $(V,\imp, \prec)$. To compare
this situation with the one $\Phi$ and $\Psi$ take care of, suppose $a \imp b
\imp c \leftarrow a$ and add $x$ and $y$ such that $a \imp x$ and $y \imp c$.
Since $c\prec a \prec b$ and $a \prec c \prec b $ cannot occur
simultaneously, only one of $x$ and $y$ can witness (if $\varphi(a,b,c)$,
resp.\ $\psi(b,c,a)$, fails) the non extendibility of $(V,\imp, \prec)$.

Despite the previous considerations the next lemma shows that $x$ witnesses
the non extendibility of $(V,\imp, \prec)$ if and only if $y$ does.

\begin{lemma} \label{SoloTheta}
Let $(V, \imp)$ be a pseudo-transitive ograph and suppose $a,b,c,d \in V$
are such that $a\imp c$, $b\imp d$, $a\mid b$ and $d\mid c$. Then
$\theta(a,b,c,d)$ holds if and only if $\sigma(a,b,c,d)$ does.

Therefore, if $(V,\imp, \prec)$ is a \ght\ then
$\Theta$ holds if and only if $\Sigma$ does.
\end{lemma}
\begin{proof}
Since the antecedents of $\Sigma$ and $\Theta$ coincide and imply the
hypothesis of the first statement, it is clear that the second statement
follows from the first.

For the forward direction of the first statement, let $e_0, \dots, e_n$
witness $\theta(a,b,c,d)$. By Remark \ref{rem:thetaimp} we can assume that $b
\imp e_i$ whenever $i>0$. We claim that $e_n, \dots, e_0$ witness
$\sigma(a,b,c,d)$. In fact, conditions ($\sigma_1$) and ($\sigma_3$) of
$\sigma(a,b,c,d)$ are exactly conditions ($\theta_3$) and ($\theta_1$) of
$\theta(a,b,c,d)$. Condition ($\sigma_2$) of $\sigma(a,b,c,d)$ is now
$\forall i< n \, (b \imp e_{i+1} \imp e_i)$ and follows easily from our
assumption on the $e_i$'s and from condition ($\theta_2$) of
$\theta(a,b,c,d)$. We are left with showing condition ($\sigma_4$) of
$\sigma(a,b,c,d)$, i.e.\ $e_0 \mid c$. Suppose on the contrary that $e_0-c$.
Since $c \mid d$, by Property \ref{Incom} it follows that $\forall i \le n \,
(c-e_i)$. In particular, $c-e_n$ and so $e_n \imp c$ because $a \mid e_n$ by
($\theta_4$) of $\theta(a,b,c,d)$. But then $c-d$ by \pt{e_n}, contrary to
the assumptions.

The proof of the backward direction is analogous.
\end{proof}

Thanks to the previous lemma it suffices to concentrate on $\Theta$.

%

The following duality principle is analogous to Remark \ref{dualityPhi}. It
is not needed elsewhere and we include it here for completeness without
proof.

\begin{remark}
Notice that the \ght\ $(V, \imp, \prec)$ satisfies $\Theta$ if and only if
the \ght\ $(V, \leftarrow, \succ)$ satisfies $\Theta$.
\end{remark}

\begin{lemma} \label{charAB}
Let $(V,\imp,\prec)$ be a \ght. Let $a,b,c,d\in V$ be such that $a\imp c$, $b\imp d$, $a\mid
b$, and $d\mid c$ and assume that $\theta(a,b,c,d)$ holds. Then for each
$(V\cup\{x\}, \imp')$ pseudo-transitive extension of $(V,\imp)$ if $a-'x-'b$
holds we have $x-'d$, and if $c-'x-'d$ holds we have $x-'b$.
\end{lemma}
\begin{proof}
Let $(V\cup\{x\}, \imp')$ be a pseudo-transitive extension of $(V,\imp)$ with
$a-'x-'b$. Notice that, since $a\mid b$, either $a \leftarrow' x \imp' b$ or
$a\imp' x \leftarrow' b$. In the first case \pt a and \pt b guarantee that
$c-'x-'d$, so we concentrate on the other case.

Suppose that $a\imp' x \leftarrow' b$ and let $e_0, \dots, e_n$ witness
$\theta(a,b,c,d)$. Towards a contradiction, assume $x \mid' d$. Notice that
$x-'e_0$ by \pt b (we use condition ($\theta_1$)). Hence by condition ($\theta_2$) and Property
\ref{Incom} it holds that $\forall i\le n \, (x-'e_i)$, so that in particular
$x-'e_n$. It cannot hold that $x \imp e_n$, otherwise $a-e_n$ by \pt x
contrary to ($\theta_4$). Hence $e_n \imp x$ holds. Moreover, $d \imp e_n$ by
condition ($\theta_3$) and so \pt{e_n} implies $d-'x$.

A similar argument shows that $c-'x-'d$ implies $x-'b$. The only change is
due to the fact that when $c \leftarrow' x \imp' d$ then we use
$\sigma(a,b,c,d)$, which holds by Lemma \ref{SoloTheta}.
\end{proof}

We can now show that $\Theta$ is sufficient for the second condition of Lemma
\ref{Nec}.

\begin{lemma}\label{char2Frecce}
Let $(V,\imp,\prec)$ be a \ght\ satisfying $\Theta$. For each $(V\cup\{x\},
\imp')$ pseudo-transitive extension of $(V,\imp)$ we have $N^-(x)
\setminus N^+(x) \prec N^+(x) \setminus N^-(x)$.
\end{lemma}
\begin{proof}
By Lemma \ref{CharProp2}.1 it suffices to prove condition $\Lambda$. Fix
$a,b,c,d \in V$ such that $a \prec c$, $d \prec b$, $a\mid b$, and $c\mid d$
and assume that $a-'x-'b$. We need to prove that either $x-'c$ or $x-'d$.

Since $a-c$ and $d-b$ there are four possible situations. If $a \imp c$ and
$d \imp b$, but $x\mid 'c$, then $a \imp' x \leftarrow' b$ and $x-'d$ follows
by \pt b. If $c \imp a$ and $b \imp d$ the argument is similar. If instead $a
\imp c$ and $b \imp d$ notice that $\Theta$ implies $\theta(a,b,c,d)$ or
$\theta(b,a,d,c)$: then Lemma \ref{charAB} yields the conclusion. The last
possibility is $c \imp a$ and $d \imp b$, where we use the second part of
Lemma \ref{charAB} (in this case $a,b,c,d$ play roles which are opposite to
those of the Lemma).
\end{proof}

We now prove that $\Theta$ is necessary for $N^-(x)\setminus N^+(x) \prec
N^+(x) \setminus N^-(x)$ if $\Phi$ and $\Psi$ hold.

\begin{lemma} \label{suf2Frecce}
Let $(V,\imp,\prec)$ be a \ght\ such that $\Phi$ and $\Psi$ hold and $\Theta$ fails. Then
there is a pseudo-transitive extension $(V\cup \{x\}, \imp')$ of $(V, \imp)$
such that $N^-(x)\setminus N^+(x) \nprec N^+(x) \setminus N^-(x)$ and hence
$(V,\imp, \prec)$ is not extendible by Lemma \ref{Nec}.
\end{lemma}
\begin{proof}
Let $a,b,c,d \in V$ be such that $a\imp c$, $b\imp d$, $a\mid b$, $c\mid d$,
$a \prec c$, $d \prec b$ and $\neg\theta(a,b,c,d)$. We fix $x \notin V$ and
define an extension $\imp'$ of $(V,\imp)$ to $V \cup \{x\}$ in stages, as an
increasing union ${\imp'} = {\bigcup_{n\in \nat} \imp_n}$. For each stage
$n$, $\imp_n$ is defined as follows:
\begin{itemize}
  \item $\imp_0$ extends $\imp$ by adding the edges $a \imp x$ e $b \imp
      x$;
  \item $\imp_{n+1}$ extends $\imp_n$ by adding edges
\[\begin{cases}
x \imp_{n+1} u & \text{ if } \exists v ((x\imp_n v \imp u ) \lor (u \imp v \imp_n x))
\text{ and either $c \imp u$ or $d \imp u$;}\\
u \imp_{n+1} x & \text{ if } \exists v ((x\imp_n v \imp u ) \lor (u \imp v \imp_n x))
\text{ and either $u \imp c$ or $u \imp d$.} \\
\end{cases}\]
\end{itemize}
Notice that $x \imp' u$ and $u \imp' x$ are incompatible, since if $c \imp u$
or $d \imp u$ then we have neither $u \imp c$ nor $u \imp d$.

If we assume that $\imp'$ is pseudo-transitive we can complete the proof as
follows. Since $\Phi$ and $\Psi$ hold, by Lemma \ref{characterization} we
have $N(x)=N^+(x)\cup N^-(x)$. On the other hand, by definition of $\imp'$,
$a-'x-'b$ but $x\mid 'c$ and $x\mid 'd$ (because $c \mid d$ and hence we
never set $x \imp_{n+1} c$ or $x \imp_{n+1} d$) and condition $\Lambda$
fails. Thus, by Lemma \ref{CharProp2}{\color{red}.}\ref{CharProp22}, $N^-(x) \setminus N^+(x) \nprec
N^+(x) \setminus N^-(x)$.

Therefore to complete the proof it suffices to check the pseudo-transitivity
of $\imp'$. We first make a few preliminary observations.

\begin{claim}\label{suf2Frecce.3}
If $v \in V$ is such that $v \imp' x$ then either $v \imp c$ or $v \imp d$.
Similarly, if $u \in V$ is such that $x \imp' u$ then either $c \imp u$ or $d
\imp u$.
\end{claim}
\begin{prf}
Let $n$ be least such that $v \imp_n x$. If $n=0$ then $v$ is either $a$ or
$b$, which satisfy the conclusion. If $n>0$ then $v \imp c$ or $v \imp d$ is
required by definition. When dealing with $u$, the case $n=0$ cannot hold.
\end{prf}

\begin{claim}\label{suf2Frecce.4}
Let us assume that for $z,w \in V$ we have either $x \imp' z \imp w$ or $w
\imp z \imp' x$. Then if $w-c$ and $w\mid d$ we have also $z-c$ and $z\mid
d$, and similarly if $w-d$ and $w\mid c$ we have also $z-d$ and $z\mid c$.
\end{claim}
\begin{prf}
Assume $w-c$ and $w\mid d$. If  $x \imp' z\imp w$, then $c \imp z$ or $d \imp
z$ by Claim \ref{suf2Frecce.3}. If $d \imp z$, then $d-w$ by \pt z, contrary
to the assumption. So $c \imp z$, while $z \imp d$ cannot hold because $c\mid
d$. Thus we have $z-c$ and $z \mid d$. If $w \imp z \imp' x$ the argument is
similar.

The second statement is proved analogously.
\end{prf}

%

\begin{claim}\label{suf2Frecce.2}
$\forall e\, (e \ne a \land e \ne b \land e-_1 x \Imp e\imp a \lor e\imp b)$
\end{claim}
\begin{prf}
Let us suppose that $e \neq a$, $e \ne b$ and $e -_1 x$, so that $e -_0 x$
does not hold. The definition of $\imp_1$ implies that for some $v$ we have
either $x\imp_0 v \imp e$ or $e \imp v \imp_0 x$. Since the only $v$'s such
that $v -_0 x$ are $a$ and $b$, and $a \imp_0 x$ and $b \imp_0 x$, we must
have the second possibility with $v$ either $a$ or $b$.
\end{prf}

In order to prove that  $\imp'$ is  pseudo-transitive, there are some cases
to  consider.

\begin{enumerate}[a.]

\item $v\imp' x\imp' u$. By Claim \ref{suf2Frecce.3} either $v \imp c$ or $v
    \imp d$ and also $c \imp u$ or $d \imp u$. If either $v \imp c \imp u$
    or $v \imp d \imp u$, then $u-v$ follows by \pt c  or \pt d  of
    $\imp$.

    We now concentrate on the case $v\imp c$ and $d \imp u$, the other
    being similar. Notice that $c \mid d$ implies that $u \imp c$ and $d
    \imp v$ do not hold. Moreover, we can assume that $v \imp d$ and $c
    \imp u$ both fail, else we are in one of the previous cases. Hence $u
    \mid c$ and $v \mid d$. If $n$ is the minimum stage such that $x
    \imp_{n+1} u$ (notice that $x \imp_0 u$ cannot happen), there exists
    $e_{n-1}$ such that $x \imp_n e_{n-1} \imp u$ or $u \imp e_{n-1} \imp_n
    x$. Analogously, there must be an $e_{n-2}$ such that  $x \imp_{n-1}
    e_{n-2} \imp e_{n-1}$ or $e_{n-1} \imp e_{n-2} \imp_{n-2} x$. Iterating
    this procedure, we get to $x -_1 e_0$. Set also $e_n=u$. Similarly, let
    $k$ be least such that $v \imp_k x$ (in this case $k=0$ is possible)
    and set $h_k=v$. If $k>0$, with a procedure similar to the one used
    before, we find $h_0, \dots, h_{k-1}$ such that $h_j$ witnesses that $x
    -_{j+1} h_{j+1}$ for each $j<k$.

    Notice that a backward induction using Claim \ref{suf2Frecce.4} easily
    entails $\forall i <n \, (e_i-d \land e_i \mid c)$ and $\forall j <k \,
    (h_j-c \land h_j \mid d)$. Notice also that for each $i<n$ either $d
    \imp e_i \imp e_{i+1}$ or $e_{i+1} \imp e_i \imp d$ holds. In fact, if
    $d \imp e_i$, then $x \imp' e_i$ by definition and so $e_i \imp
    e_{i+1}$ by choice of $e_i$. If $e_i \imp d$ the argument is similar.
    Arguing as in the previous lines it is easy to show that for each $j<
    k$ either $c \imp h_j \imp h_{j+1}$ or $h_{j+1} \imp h_j \imp c$ holds
    as well.

Let $i \le n$ be least such that $d \imp e_i$. We claim that $e_0, \dots,
e_i$ satisfy the first three conditions of $\theta(a,b,c,d)$:
\begin{enumerate}[($\theta_1$)]
\item $e_0 \imp b$ by Claim \ref{suf2Frecce.2} because $e_0 \imp a$
    implies $e_0 - c$ by \pt a, which contradicts the above observation;
\item $\forall j<i \,(e_{j+1} \imp e_j \imp d)$: this is immediate by the minimality of $i$ and the observation in the previous paragraph;
\item $d \imp e_i$ by choice of $i$;
\end{enumerate}
Since $\theta(a,b,c,d)$ fails, condition ($\theta_4$) must fail, i.e.\ we
have $e_i-a$.

Since $a \mid d$ we can apply Property \ref{Incom} to obtain that $e_j-a$
for every $j\le n$ with $j \ge i$.
Recalling that $e_n=u$, we obtained $u-a$:
then $a \imp u$ because $a \mid d$.

We show that $h_j-u$ for every $j \le k$. Arguing as in the proof of
($\theta_1$) above, we have $h_0 \imp a$ so that $h_0-u$ by \pt a. Thus,
since $u \mid c$, we can apply Property \ref{Incom} again to obtain the
desired conclusion.
Recalling that
$h_k=v$ we have obtained $u-v$.

\item $x\imp' v\imp u$ then $c \imp v$ or $d \imp v$ by Claim
    \ref{suf2Frecce.3}. By \pt v, either $c-u$ or $d-u$ and $u$ satisfies
    one of the conditions in the definition of $\imp'$. Thus $x-'u$.
\item $u\imp v\imp' x$ is similar to the previous item.
\end{enumerate}
This shows that $\imp'$ is pseudo-transitive and hence that $(V,\imp, \prec)$ is not
extendible.
\end{proof}

Summarizing, we obtained a characterization of the conditions of Lemma
\ref{Nec}.

\begin{theorem}\label{N2<N1}
Let $(V,\imp,\prec)$ be a \ght. The following are equivalent:
\begin{enumerate}
\item for each pseudo-transitive extension $(V\cup \{x\}, \imp')$ of
    $(V,\imp)$ both $N(x)=N^+(x)\cup N^-(x)$ and $N^-(x)\setminus N^+(x)
    \prec N^+(x)\setminus N^-(x))$ hold;
\item $\Phi$, $\Psi$ and $\Theta$ are satisfied.
\end{enumerate}
\end{theorem}
\begin{proof}
The implication (1) $\Imp$ (2) follows from Lemmas \ref{sufTrian} and
\ref{suf2Frecce}. The implication (2) $\Imp$ (1) follows from Lemmas
\ref{characterization} and \ref{char2Frecce}.
\end{proof}

Thanks to Theorem \ref{N2<N1} we can now reformulate Lemma \ref{Nec} in a way
that does not refer to all possible pseudo-transitive extensions of
$(V,\imp)$ but mentions only structural properties of $(V, \imp)$ and
$\prec$.

\begin{theorem}
Let $(V,\imp,\prec)$ be an extendible \ght. Then $\Phi$, $\Psi$ and $\Theta$ are
satisfied.
\end{theorem}

It follows from Lemma \ref{trans} below that the reverse implication holds as
well, namely that $\Phi$, $\Psi$ and $\Theta$ are also sufficient conditions
of the extendibility of a \ght.

\section{The Smart Extension Algorithm}\label{sec:algo}

In this section we define an on-line algorithm to transitively reorient a
countable pseudo-transitive ograph. Before defining the algorithm we give
some preliminary definitions.

\begin{definition}
Let $(V,\imp,\prec)$ be a \ght. If $(V\cup \{x\},\imp')$ is a pseudo-transitive extension of $\imp$,
we define inductively the following subsets of $N(x)$:
\begin{align*}
     S_0^-(x)&=  N^-(x)\setminus N^+(x); \\
     S_0^+(x) &=  N^+(x)\setminus N^-(x);  \\
     S_i(x) &= S_i^-(x) \cup S_i^+(x); \\
   S_{i+1}^-(x)&=  \{a\in N(x)\setminus \textstyle{\bigcup_{j\le i} S_j(x)} \, \mid \,  \exists s \in S_i^-(x) (a\mid s)\};   \\
   S_{i+1}^+(x)&=  \{a\in N(x)\setminus \textstyle{\bigcup_{j\le i} S_j(x)} \, \mid \,  \exists s \in S_i^+(x) (a\mid s)\}.
\end{align*}
Let $S^+(x)=\bigcup_{i\in \nat} S_i^+(x)$, $S^-(x)=\bigcup_{i\in \nat}
S_i^-(x)$ and $S(x)= S^-(x) \cup S^+(x) = \bigcup_{i\in \nat} S_i(x)$. Let
also $T(x)=N(x)\setminus S(x)$.

If $* \in \{+,-\}$ we say that a sequence $\rho = \langle \rho(0), \rho(1),
\dots, \rho(|\rho|-1) \rangle$ of elements of $V$ is a $*$-sequence if
$\rho(i) \in S_i^*(x)$ for every $i < |\rho|$ and $\rho(i) \mid \rho(i+1)$
for every $i < |\rho|-1$.
\end{definition}

\begin{remark}\label{remS*}
If $N(x) = N^+(x) \cup N^-(x)$ then $S(x) \setminus S_0(x)$ and $T(x)$ are
both included in $N^+(x)\cap N^-(x)$. Moreover $S(x) \subseteq N(t)$ for
every $t\in T(x)$ (because if $t \in N(x)$ and $S_i(x) \setminus N(t) \neq
\emptyset$ then $t \in S_{i+1}(x)$).

Notice that if $s \in S_i^*(x)$ then there exists a $*$-sequence $\rho$ such
that $\rho(i)=s$.
\end{remark}

For the remainder of the section we use $S^*(x)$ as a shorthand for either
$S^+(x)$ or $S^-(x)$. $S^*_i(x)$ is used similarly and $s_i^*$ always denotes
an element of $S_i^*(x)$. We now prove some properties of $S^*(x)$ and its
subsets.

\begin{property}\label{S^*}
Let $(V,\imp,\prec)$ be a \ght. Let $(V\cup \{x\},\imp')$ be a pseudo-transitive extension of
$\imp$.
\begin{enumerate}
\item\label{s^*} Fix $v \in V \cup \{x\}$ and $* \in \{+,-\}$. If $\rho$ is
    a $*$-sequence such that $\forall i< |\rho|\, (v-'\rho(i))$ then either
    $\forall i < |\rho| \,(\rho(i) \imp' v)$ or $\forall i < |\rho| \,(v
    \imp' \rho(i))$.
\item\label{SnonS} $S^-(x) \prec T(x)$ and $T(x) \prec S^+(x)$.
\item\label{5.5} If $\rho^*$ is a $*$-sequence for $* \in \{+,-\}$,
    $\rho^+(0) \imp' x \leftarrow' \rho^-(0)$ and $e_0, \dots, e_n$ witness
    $\varphi(\rho^-(0),\rho^+(0),f)$ for some $f \mid' x$, then there
    exists $i \le n$ such that $e_i, \dots, e_n$ witness
    $\varphi(\rho^-(k),\rho^+(j),x)$, for each $k$ and $j$. Moreover,
    $\rho^-(k) \imp' x\leftarrow' \rho^+(j)$. The same statement holds with
$\psi$ in place of $\varphi$.
\end{enumerate}
\end{property}
\begin{proof}
(1) is obvious by pseudo-transitivity of $\imp'$.\smallskip

To prove (2) we fix $t \in T(x)$ and prove by induction on $i$ that $t \prec
S^+_i(x)$ for every $i$. For the base of the induction, $t \prec S_0^+(x)$
follows from $S_0^+(x) \subseteq N(t)$ (Remark \ref{remS*}), $t \in N^-(x)$
and $S^+_0(x) \cap N^-(x) = \emptyset$. For the induction step let $s^+_{i+1}
\in S^+_{i+1}(x)$ and choose $s^+_i \in S^+_i(x)$ such that $s^+_{i+1} \mid
s^+_i$. By induction hypothesis $t \prec s^+_i$ and hence, since
$s^+_{i+1}-t$ (again by Remark \ref{remS*}), we have $t \prec s^+_{i+1}$.
This shows $T(x) \prec S^+(x)$. Analogously we prove $S^-(x) \prec
T(x)$.\smallskip

To prove (3) fix $\rho^+$, $\rho^-$, $f$, $e_0, \dots, e_n$ satisfying the
hypothesis. Let $m^*$ be the length of $\rho^*$ for $* \in \{+,-\}$. We write
$s_k^*$ in place of $\rho^*(k)$. Since $s_0^* \imp' x$ and $S(x) \subseteq
N(x)$, (\ref{s^*}) implies that $s_k^* \imp' x$ for each $k<m^*$.

Applying Property \ref{ei-x} to $(V\cup \{x\},\imp')$ we obtain that there
exists $i \le n$ such that  $e_i, \dots, e_n$ witness
$\varphi(s_0^-,s_0^+,x)$. For the sake of convenience assume $i=0$, so that
$e_0, \dots, e_n$ witness $\varphi(s_0^-,s_0^+,x)$ as well.

Fix $* \in \{+,-\}$. We claim that $\forall k<m^* \,\forall i\le n\, (e_i -
s_k^*)$. The proof is by double induction\footnote{since we are dealing with
formulas where all quantifiers are bounded, this can be done within $\rca$.}.
Suppose $\forall i\le n\, (e_i - s_\ell ^*)$ for each $\ell < k$. We prove by
induction on $i$ that $\forall i\le n \,(e_i - s_k^*)$. For the base case,
$e_0 - s_k^*$ by \pt x\ since $s_k^* \imp' x$ and $x \imp' e_0$ by
($\varphi_1$) of $\varphi(s_0^-,s_0^+,x)$. For the induction step suppose
$e_i - s_k^*$. If $s_k^* \imp e_i$, then $s_0^* \imp e_i$ by (\ref{s^*})
(that applies because $\forall \ell < k\, (e_i - s^*_\ell)$). Then $e_i \imp
e_{i+1} $ by $\varphi(s_0^-,s_0^+,x)$. Hence $e_{i+1} - s_k^*$ by \pt{e_i}.
If $e_i\imp s_k^*$, the argument is analogous.

Let $k<m^-$ and $j<m^+$. We check that the three conditions of
$\varphi(s_k^-,s_j^+,x)$ are satisfied. Condition ($\varphi_1$) holds
trivially since it coincides with ($\varphi_1$) of $\varphi(s_0^-,s_0^+,x)$.
To check that ($\varphi_2$) holds suppose $s_k^- \imp e_i$. Then $s_0^- \imp
e_i$ by (\ref{s^*}) and thus $s_0^+ \imp e_i \imp e_{i+1}$ by ($\varphi_2$)
of $\varphi(s_0^-,s_0^+,x)$. By (\ref{s^*}) again it holds that $s_j^+ \imp
e_i$ holds as well. An analogous argument shows that if $e_i \imp s_k^-$,
then $e_{i+1} \imp e_i \imp s_j^+$. These establish that ($\varphi_2$) of
$\varphi(s_k^-,s_j^+,x)$ holds. Condition ($\varphi_3$) is checked in a
similar way.
\end{proof}

Notice that Property \ref{S^*}.\ref{SnonS} implies $S^-(x) \prec S^+(x)$
whenever $T(x) \neq \emptyset$. To see that this holds in general we need to
strengthen the hypothesis on the reorientation of $(V,\imp)$.

\begin{lemma}\label{CompS+-}
Let $(V,\imp,\prec)$ be a \ght\ such that $\Psi$, $\Phi$ and $\Theta$ are satisfied. Let
$(V\cup \{x\},\imp')$ be a pseudo-transitive extension of $\imp$. Then
$S^-(x) \prec S^+(x)$ and hence $S^-(x) \cap S^+(x) = \emptyset$.
\end{lemma}
\begin{proof}
Let $s^-\in S^-(x)$. We first claim that $s^- - s^+$ for every $s^+ \in
S^+(x)$, which is obviously necessary for $S^-(x) \prec S^+(x)$. Since $s^-,
s^+ \in N(x)$, there are four possibilities.

If $s^+ \imp' x\imp' s^-$ or $s^- \imp' x\imp' s^+$, then by \pt x we have
$s^- - s^+$.

Otherwise, $s^- \imp' x \leftarrow' s^+$ or $s^+ \leftarrow' x \imp' s^-$.
Suppose the former holds. For $* \in \{+,-\}$ choose a $*$-sequence $\langle
s_0^*, \dots, s_{m^*}^* \rangle$ such that $s_{m^*}^*= s^*$.
Recall that, by definition of $*$-sequence, $s_i^* \in S_i^*(x)$ for each $i
\le m^*$ and $s_i^* \mid s_{i+1}^*$ for each $i< m^*$.

Since $s^- \imp' x \leftarrow' s^+$, Property
\ref{S^*}{\color{red}.}\ref{s^*} implies that $s_0^+ \imp' x \leftarrow'
s_0^-$. Since $s^+_0 \notin N^-(x)$, there exists $f$ such that $f \prec
s^+_0$ and $f\mid' x$. Analogously, there exists $e$ such that $s^-_0 \prec
e$ and $e\mid' x$. Given that $f\mid' x\mid' e$ and $s_0^+ \imp' x
\leftarrow' s_0^-$, then $ s^+_0 \imp f$ and $s_0^- \imp e$. Moreover, since
$s_0^- \prec s_0^+$ by Theorem \ref{N2<N1}, it holds  $s_0^+ \imp s_0^-$ or
$s_0^- \imp s_0^+$. Suppose the latter, the other case being similar using
$e$ in place of $f$. We have $s_0^- - f$ by \pt{s_0^+}, and thus $s_0^- \imp
f$ since $s_0^- \imp' x$ and $x \mid' f$. Since $s^-_0 \in N^-(x)$, then
$s_0^- \prec f$. Summarizing, we have just shown that $s_0^- \prec f \prec
s_0^+$  and $s_0^+ \imp f \leftarrow s_0^-$. Since we are assuming $\Phi$
holds, there are $e_0, \dots, e_n$ witnessing $\varphi(s_0^-,s_0^+,f)$.
Applying Property \ref{S^*}{\color{red}.}\ref{5.5} we obtain that there
exists an $i \le n$ such that  $e_i, \dots, e_n$ witness
$\varphi(s_k^-,s_j^+,x)$ for each $k \le m^+$ and $j \le m^-$. In particular
$\varphi(s^-,s^+,x)$ is satisfied and so either $s^- \imp e_n\imp s^+$ or
$s^+ \imp e_n \imp s^-$ holds. In both cases, by \pt{e_n}, $s^+ - s^-$ as we
wanted to show.

If instead $s^- \leftarrow' x \imp' s^+$ the argument is similar, reversing
all arrows and using $\Psi$.

We have thus established our claim that $s^- - s^+$ for every $s^+ \in
S^+(x)$. Now we prove by induction on $i$ that $s^- \prec S^+_i(x)$ for every
$i$. For the base of the induction, $s^- \prec s^+_0$ for every $s^+_0 \in
S_0^+(x)$ because $s^- - s^+_0$, $s^- \in N^-(x)$ and $s^+_0(x) \notin
N^-(x)$. For the induction step let $s^+_{i+1} \in S^+_{i+1}(x)$ and choose
$s^+_i \in S^+_i(x)$ such that $s^+_{i+1} \mid s^+_i$. By induction
hypothesis $s^- \prec s^+_i$ and hence, since $s^+_{i+1}-s^-$, we have $s^-
\prec s^+_{i+1}$.
\end{proof}

These relations between subsets of $N(x)$ explain the choices for the
reorientation of $V \cup \{x\}$ made in the following definition.

%
%
%

\begin{definition}\label{def:<'}
Let $(V,\imp,\prec)$ be a \ght\ satisfying $\Phi$,
$\Psi$ and $\Theta$ and such that $V\subseteq \nat$.
Let $(V\cup \{x\},\imp')$ be a pseudo-transitive
extension of $\imp$.

We define $\prec'$, the \emph{smart extension of $\prec$ to $(V\cup
\{x\},\imp')$}, as the binary relation that extends $\prec$ to $V \cup \{x\}$
by establishing the relationship between $x$ and each $v \in V$ recursively
as follows:
\begin{enumerate}[\qquad(1)]
  \item if $v \notin N(x)$ let $v \nprec' x$ and $x \nsucc' v$;
  \item if $v \in S(x)$ then
      \begin{enumerate}[(a)]
      \item if $v \in S^-(x)$ let $v \prec' x$,
      \item if $v \in S^+(x)$ let $x \prec' v$;
    \end{enumerate}
  \item if $v \in T(x)$ then
    \begin{enumerate}[(a)]
      \item if there exists $u< v$ such that $v \prec u \prec' x$ let $v
          \prec' x$,
      \item if there exists $u< v$ such that $x \prec' u \prec v$ let $x
          \prec' v$,
      \item otherwise let $v \prec' x$ if $v \imp' x$ and $x \prec' v$ if
          $x \imp ' v$.
    \end{enumerate}
\end{enumerate}
\end{definition}

Notice that $\prec'$ depends on the order $<$ on $\nat$, is always a
reorientation of $(V \cup \{x\}, \imp')$, and extends $\prec$.

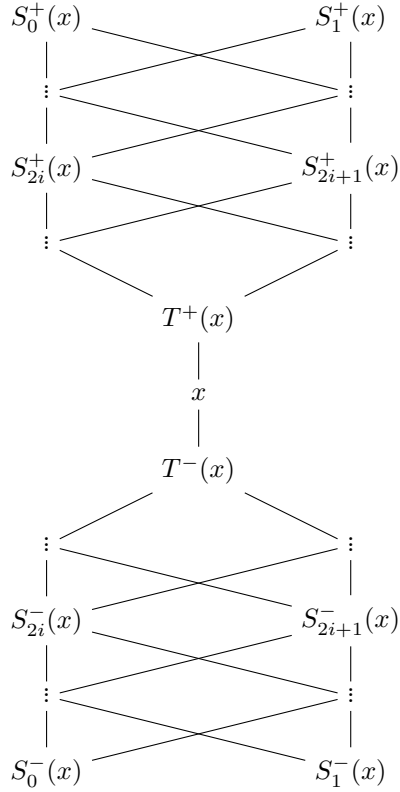
\begin{figure}
\begin{tikzpicture}
\node (1) at (0,0){$S_0^-(x)$};
\node (2) at (0,1) {\rvdots};
\node (3) at (0,2){$S_{2i}^-(x)$};
\node (4) at (0,3) {\rvdots};
\node (5) at (2,4){$T^-(x)$};
\node (6) at (2,5){$x$};
\node (7) at (2,6){$T^+(x)$};
\node (8) at (0,7) {\rvdots};
\node (9) at (0,8){$S_{2i}^+(x)$};
\node (10) at (0,9) {\rvdots};
\node (11) at (0,10){$S_0^+(x)$};

\node (a) at (4,0){$S_1^-(x)$};
\node (b) at (4,1) {\rvdots};
\node (c) at (4,2){$S_{2i+1}^-(x)$};
\node (d) at (4,3) {\rvdots};
\node (e) at (4,7) {\rvdots};
\node (f) at (4,8){$S_{2i+1}^+(x)$};
\node (g) at (4,9) {\rvdots};
\node (h) at (4,10){$S_1^+(x)$};

\draw  (1) --(2);
\draw  (2) --(3);
\draw  (3) --(4);
\draw  (4) --(5);
\draw  (5) --(6);
\draw  (6) --(7);
\draw  (7) --(8);
\draw  (8) --(9);
\draw  (9) --(10);
\draw  (10) --(11);

\draw  (a) --(b);
\draw  (b) --(c);
\draw  (c) --(d);
\draw  (d) --(5);
\draw  (7) --(e);
\draw  (e) --(f);
\draw  (f) --(g);
\draw  (g) --(h);

\draw  (1) --(b);
\draw  (2) --(c);
\draw  (3) --(d);
\draw  (8) --(f);
\draw  (9) --(g);
\draw  (10) --(h);

\draw  (a) --(2);
\draw  (b) --(3);
\draw  (c) --(4);
\draw  (e) --(9);
\draw  (f) --(10);
\draw  (g) --(11);
\end{tikzpicture}
\caption{\small A smart extension.}\label{figure<'}
\end{figure}
For a visual understanding of $\prec'$ see Figure \ref{figure<'}. Here we
denote by $T^-(x)$, resp.\ $T^+(x)$, the subset of $T(x)$ consisting of the
vertices which are below, resp.\ above, $x$. Moreover the picture shows that
$S^-_i(x) \prec S^-_{i+2}(x)$ and $S^+_{i+2}(x) \prec S^+_i(x)$: we leave to
the reader to prove these relations, since we do not need them. The picture
may suggest that $(N(x),\prec)$ has width two, but this is not the case
because there may be nontrivial antichains within some $S_i^*(x)$ and/or
$T^*(x)$.

The hypothesis that $(V, \imp,\prec)$  satisfies $\Phi$, $\Psi$ and $\Theta$
makes sure that Conditions (2a) and (2b) of Definition \ref{def:<'} are
mutually exclusive, by Lemma \ref{CompS+-}. Some of the clauses of Definition
\ref{def:<'} are necessary for $\prec'$ to be a transitive reorientation of
$\imp'$. Condition (1) is obviously necessary for $\prec'$ to be a
reorientation. The choice $S^-(x) \prec x$ made by Condition (2a) is
explained by an inductive argument: $S^-_0(x) \prec' x$ is required because
$S^-_0(x) \cap N^+(x) = \emptyset$, and if $S^-_i(x) \prec x$ then the
members of $S^-_{i+1}(x)$ (each incomparable with some element of $S^-_i(x)$)
cannot lie above $x$. The same argument applies to $S^+(x)$ and justifies
Condition (2b). Conditions (3a) and (3b) are clearly necessary for
transitivity. Condition (3c) is applied when the relationship between $x$ and
$v_i$ is not decided by the previous conditions and in this case $\prec'$
simply preserves the direction of $\imp'$. \smallskip

From a complexity point of view, defining the sets $S^+(x)$ and $S^-(x)$
requires more resources than setting the relation between $x$ and $v \in V$
according to Definition \ref{def:<'}. The sets $S_0^+(x)$ and $S_0^-(x)$ are
computed in at most $|V^2|$ steps, since one needs to consider each $v \in
N(x)$ and for each such $v$ to go through each $u \in N(v) \setminus N(x)$.
The remaining members of $S^+(x)$ and $S^-(x)$ can be found by a depth-first
search algorithm applied to the non-adjacency graph $(V\cup \{x\},E')$ (the
complexity of depth-first search algorithm is $O(|V| + |E|)$, see
\cite[Section 22.3]{Cormen}). To this end notice that for each $s_0 \in
S_0^+(x)$ there exists a sequence\footnote{Such sequences may not be
$+$-sequences, because it may be the case that $v_i \in S_j^+(x)$, for some
$j < i$ due to the incomparability chain caused by some other element of
$S_0^+(x)$.} $v_0=s_0, v_1, \dots, v_n$, for some $n < |V|$, such that $v_i
\E' v_{i+1}$ and $v_i \in N(x)$, for each $i \le n$. Then each $v_i \in
S^+(x)$. The same obviously applies also to $S^-(x)$.

Therefore an upper bound for the complexity of the smart extension is
$O(|V|^2)$.


\begin{definition}
Let $(V,\imp)$ be a pseudo-transitive ograph with $V$ an initial interval of
$\nat$. The relation $\prec$ is the \emph{smart reorientation} of $\imp$ if
it at each step $s$ the reorientation ${\prec_{s+1}} = {\prec\upharpoonright
\{0, \dots, s\}}$ is obtained as the smart extension of $\prec_s$.
\end{definition}

\begin{remark}\label{duality}
Notice that the smart reorientation of
$(V,\leftarrow)$ is the reversal of the smart reorientation of $(V,\imp)$.
\end{remark}

Theorem \ref{MainTheorem} proves that the smart reorientation algorithm is
correct. To obtain this result we prove some properties of smart
reorientations. In particular we introduce the notion of \lq lazy
reorientation\rq\ in Definition \ref{DefLazy}. The intuitive idea behind it
is the following one: an edge $a \imp b$ is reversed only when this is really
needed to obtain a transitive reorientation, because $a \imp b \imp c \imp a$
and the edges $b \imp c$ and $c \imp a$ are not reversed.

\begin{property}\label{Minimoj}
Let $(V,\imp,\prec)$ be a \ght\ with $V \subseteq \nat$. Let
 $(V\cup \{x\},\imp')$ be a pseudo-transitive extension of $\imp$. Let $\prec'$ be the
smart extension of $\prec$.

If $a \prec' x$ because we applied condition (3a) with witness $b$ then $b
\in T(x)$. Moreover we can choose $b$ so that $b \imp' x$.

Similarly, if $x \prec' a$ because we applied condition (3b) with witness $b$
then $b \in T(x)$ and we can assume $x \imp' b$.
\end{property}
\begin{proof}
Let $a\in T(x)$ and $b$ with $b < a$ be such that $a \prec b \prec' x$. Since
$b \prec' x$ then $b \in S^-(x) \cup T(x)$. But $b \notin S^-(x)$ by Property
\ref{S^*}{\color{red}.}\ref{SnonS} and hence $b \in T(x)$. Let $b$ be least
(as a natural number) such that $a \prec b \prec' x$. If $x \imp' b$, then we
used condition (3a) when dealing with $b$ and so there exists $c< b$ such
that $b \prec c \prec' x$, contrary to the minimality of $b$. Hence, $b \imp'
x$.

The proof of the second statement is analogous.
\end{proof}

\begin{definition} \label{DefLazy}
Let $(V,\imp)$ be a pseudo-transitive ograph. The reorientation $\prec$ of
$\imp$ is a \emph{lazy reorientation} if it satisfies the following property:
for each $a,b \in V$ such that $a \imp b$ and $b \prec a$ there exists $c\in
V$ such that $b \imp c \imp a$ (i.e.\ $abc$ is a non transitive triangle), $b
\prec c \prec a$, and $c < \min (a,b)$.

$(V,\imp,\prec)$ is a \emph{\lt} if $(V,\imp)$ is a pseudo-transitive ograph
and the reorientation $\prec$ of $\imp$ is a lazy reorientation.
\end{definition}

Notice that a \lt\ is not necessarily a \ght, because we are not requiring
$\prec$ to be transitive.

\begin{remark}\label{dualityLazy}
$(V,\imp,\prec)$ is a \lt\ if only if  $(V,\leftarrow,\succ)$ is a \lt, where
$(V,\leftarrow)$ is the reverse ograph of $(V,\imp)$.
\end{remark}

\begin{property} \label{TriangNonTrans}
Let $(V,\imp)$ be a pseudo-transitive ograph with $V \subseteq \nat$ and let
$\prec$ be the smart reorientation of $(V, \imp)$. Assume that $\prec$ is
transitive, so that $(V, \imp,\prec)$ is a \ght. Then $\prec$ is lazy, i.e.\ $(V, \imp,\prec)$ is a \lt.
\end{property}
\begin{proof}
The proof of the laziness condition for every $a,b \in V$ is by induction on
the lexicographic order of the pair of natural numbers $(\max(a,b),
\min(a,b))$. Suppose $a \imp b$ and $b \prec a$ and assume that for each $a'$
and $b'$ such that $a' \imp b'$, $b' \prec a'$ and either $\max(a',b') < \max
(a,b)$ or $\max(a',b') = \max (a,b)$ and $\min(a',b') < \min (a,b)$ there
exists $c'$ such that $b' \imp c' \imp a'$, $b' \prec c' \prec a'$, and $c' <
\min (a',b')$.

By remark \ref{duality} we can assume without loss of generality that $a <
b$. According to Definition \ref{def:<'} either $a \in S^+(b)$ or $a \in
T(b)$.\smallskip

If $a \in S^+(b)$ let $i$ be such that $a \in S_i^+(b)$. We first show that
$i>0$ is impossible. If $a \in S^+_i(b)$ with $i>0$ let $\rho$ be a
$+$-sequence of length at least $2$ such that $\rho(i) = a$. By Property
\ref{S^*}{\color{red}.}\ref{s^*} we have $\rho(1) \imp b$. Since $\rho(1) \in
S^+_1(b)$, there exists $d < b$ such that $d \in S^+_0(b)$ and $d \mid
\rho(1)$. Then $b \prec d$ and $d \imp b$. Since $d \in S_0^+(b)$ there
exists $f < b$, $f \mid b$, $f \prec d$. Thus $d \imp f$. As $\max (d,f) < b
= \max (a,b)$ we can apply the induction hypothesis and there exists $c$ such
that $f \imp c \imp d$ and $f \prec c \prec d$. We have $c-b$ by \pt d, and
hence $b\imp c$ because $b \mid f$. But now $\rho(1)-c$ by \pt b, and hence
$c \imp \rho(1)$ since $\rho(1) \mid d$. Using again \pt c we have
$\rho(1)-f$, a contradiction with $\rho(1) \in S^+_1(b) \subseteq N^-(b) \cap
N^+(b)$ as $f \notin N(b)$.

Thus $i=0$ and $a \notin N^-(b)$. In particular there exists $f < b$, $f\mid
b$, $f \prec a$ and so $a \imp f$. As $\max (a,f) < \max (a,b)$ we can apply
the induction hypothesis and there exists $c < \min(a,f)$ such that $f \imp c
\imp a$ and $f \prec c \prec a$. We have $c-b$ by \pt a, and hence $b\imp c$
because $b \mid f$. Hence $b \imp c \imp a$, $b \prec c \prec a$ (because
$\prec$ is transitive and $b \mid f$), and $c < \min (a,b)$, as
required.\smallskip

If $a \in T(b)$ we applied condition (3b) of Definition \ref{def:<'} to set
$b \prec a$. Hence, by Property \ref{Minimoj} there exists $c$ such that $c<
a$, $b \prec c \prec a$ and $b \imp c$. We can assume that $c$ is least (as a
natural number) with these properties. If $c \imp a$ we have our conclusion.
We now rule out the possibility that $a \imp c$. If this was the case, by
induction hypothesis (as $\max(a,c) < \max(a,b)$) there exists $d < \min
(a,c)$ such that $c \imp d \imp a$ and $c \prec d \prec a$. By transitivity
of $\prec$ we have $b \prec d \prec a$ and $b-d$. If $b \imp d$ then $d < c$
violates the minimality of $c$. If $d \imp b$ then by induction hypothesis
(as $\max(d,b) = \max (a,b)$ and $\min(d,b) < \min(a,b)$) there exists $e <
d$ such that $b \prec e \prec d$ and $b \imp e \imp d$. But then $e < c$, $b
\prec e \prec a$ (by transitivity of $\prec$) and $b \imp e$ contradict the
minimality of $c$.
\end{proof}

\begin{lemma} \label{PsiPhiThetaSigma}
Let $(V,\imp,\prec)$ be a \ght\ which is also a \lt\ and such that $V \subseteq \nat$. Then $\Phi$,
$\Psi$ and $\Theta$ are satisfied.
\end{lemma}
\begin{proof}
Thanks to laziness checking that $\Theta$ holds is straightforward. In fact,
suppose $a \imp c$, $b \imp d$, $a \prec c$ and $d \prec b$ for some $a,b,c,d
\in V$. Since $b \imp d$ but $d \prec b$, there exists an $e_0$ such that $d
\imp e_0 \imp b$. It is immediate to check that $e_0$ witnesses
$\theta(a,b,c,d)$\footnote{\, Notice that laziness implies a strong form of
$\Theta$, in fact the sequence witnessing the formulas have always length
one.}.\smallskip

To check that $\Phi$ holds let $a,b,c \in V$ be such that $a \imp c
\leftarrow b$ and $a \prec c \prec b $. Since $b \imp c$, $c \prec b$ and
$\prec$ is lazy, there exists $e_0\in V$ such that $c \imp e_0 \imp b$,
$c\prec e_0 \prec b$ and $e_0 < \min (b,c)$. By transitivity of $\prec$ it
follows that $a \prec e_0$ and thus $a - e_0$, since $\prec$ is a
reorientation. If $a\imp e_0$, it is immediate to check that $e_0$ witnesses
$\varphi(a,b,c)$.

Otherwise $e_0 \imp a$ and, since $a \prec e_0$, by laziness there exists
$e_1\in V$ such that $a \imp e_1 \imp e_0$, $a \prec e_1 \prec e_0$ and $e_1
< \min (e_0,a)$. Notice that even if $a \imp c \imp e_0$ and  $a \prec c
\prec e_0$ it must be $c \ne e_1$ because $e_1 < e_0 < c$ by construction. By
transitivity we get that $e_1 \prec b$ and so either $e_1 \imp b$ or $b \imp
e_1$. If the former holds then $e_0, e_1$ witness $\varphi(a,b,c)$.

We have now to analyze the case when $b \imp e_1$. Since $e_1 \prec b$ by
laziness there exists $e_2$ such that $e_1 \imp e_2 \imp b$, $e_1 \prec e_2
\prec b$ and $e_2 < \min (b,e_1)$. By transitivity it holds that $a \prec
e_2$. If $a \imp e_2$, it is easy to check that $e_0, e_1, e_2$ witness
$\varphi(a,b,c)$. Otherwise $e_2 \imp a$ and we can apply laziness again to
obtain $e_3$.

This procedure provides a $<$-decreasing sequence $(e_i)$ such that $a \imp
e_{i+1} \imp e_i$ when $i$ is even, and $e_i \imp e_{i+1} \imp b$ when $i$ is
odd. The sequence stops with $e_n$ such that $a \imp e_n \imp b$. We claim
that $e_0, \dots, e_n$ witness $\varphi(a,b,c)$. In fact ($\varphi_1$) is guaranteed by
$c \imp e_0$. Moreover, for each $i<n$ either $a \imp e_i \leftarrow b$ or $a
\leftarrow e_i \imp b$ by assumption. If the former is the case then $e_i
\imp e_{i+1}$, while if the latter holds $e_{i+1} \imp e_i$ by construction.
These two facts guarantee that ($\varphi_2$) is satisfied as well. The vertex $e_n$
satisfies condition ($\varphi_3$) by construction.\smallskip

It is now easy to check that $\Psi$ is satisfied as well applying the duality
principle of Remark \ref{dualityPhi}. Consider the graph $(V, \leftarrow)$
and the transitive reorientation $\succ$. Remark \ref{dualityLazy} guarantees
that $\succ$ is lazy as well. Hence, $\Phi$ holds by what we have just shown.
Then, by Remark \ref{dualityPhi}, $\Psi$ holds in $(V, \imp)$ and $\prec$.
\end{proof}

\begin{lemma}\label{trans}
Let $(V,\imp,\prec)$ be a \ght\ such that $V \subseteq \nat$. Assume $\Phi$, $\Psi$
and $\Theta$ are satisfied. Let $(V\cup\{x\}, \imp')$ be a pseudo-transitive
extension of  $(V,\imp)$. Then the smart extension $\prec'$ to $\imp'$ is
transitive.
\end{lemma}
\begin{proof}
To check that $\prec$ is transitive, we have to consider the following cases,
where $a,b \in V$:
\begin{enumerate}
\item $a\prec' x \prec' b$. Obviously $a,b\in N(x)$ and, if $a \in S(x)$
    then $a \in S^-(x)$ while if $b \in S(x)$ then $b\in S^+(x)$. We
    consider four possibilities:
\begin{enumerate}
\item $a \in S^-(x)$, $b\in S^+(x)$: then $a \prec b$ follows from Lemma
    \ref{CompS+-}.
\item $a, b \in T(x)$: if $a \imp' x\imp' b$ or $b\imp' x \imp' a$, then
    $a-b$ by pseudo-transitivity. So we are left to check that $b\nprec
    a$. Suppose  $b \prec a$. Then, according to the definition of
    $\prec'$, if $b< a$, then $x \prec' b$ entails $x \prec' a$, while if
    $a< b$, then $a\prec' x$ entails $b \prec' x$.

Otherwise, $a\imp' x\leftarrow' b$ or $a \leftarrow'  x \imp' b$. Suppose
the latter holds, the former being similar. Since $x\imp' a$, but $a
\prec' x$ by assumption, there is, by Property \ref{Minimoj}, $c< a$ such
that $c\in T(x)$, $c \imp' x$ and $a \prec c \prec' x$. Notice that $c -
b$ by \pt x. We claim that $b \nprec c$. Suppose $b \prec c$. If $c< b$,
then, since $b \prec c \prec' x$, then $b \prec' x$ by definition,
contrary to the assumption. Otherwise, $b<c$; then, since $x \prec' b
\prec c$, then $x \prec' c$, contrary to the assumption. Thus it must be
$c \prec b$ and so $a \prec b$ because $\prec$ is transitive by
hypothesis.
\item $a\in S^-(x)$, $b\in T(x)$: $a \prec b$ follows by Property
    \ref{S^*}{\color{red}.}\ref{SnonS}.
\item $a\in T(x)$, $b\in S^+(x)$: $a \prec b$ follows by Property
    \ref{S^*}{\color{red}.}\ref{SnonS}.
	\end{enumerate}
\item $a \prec b \prec' x$. Since $b \in S^-(x) \cup T(x)$ we have $b \in
    N^-(x)$ and thus $a \in N(x)$. If $a \in S(x)$, Property
    \ref{S^*}{\color{red}.}\ref{SnonS} or Lemma \ref{CompS+-} imply $a \in S^-(x)$, and
    thus $a \prec' x$.

    If instead $a \in T(x)$ then $b \in T(x)$ by Property
    \ref{S^*}{\color{red}.}\ref{SnonS}. If $b<a$ then $a \prec b \prec' x$ implies $a
    \prec' x$. If $a<b$ then $x \prec' a$ would imply $x \prec' b$; hence
    $a\prec' x$ holds also in this case.
\item $x \prec' a \prec b$. The argument is similar to the previous
    case.\qedhere
\end{enumerate}
\end{proof}

The following theorem proves that Definition \ref{def:<'} provides an
algorithm to transitively reorient pseudo-transitive graphs.

\begin{theorem}\label{MainTheorem}
Let $(V,\imp)$ be a pseudo-transitive ograph with $V$ an initial interval of
$\nat$ and let $\prec$ be the smart reorientation of $(V, \imp)$. Then
$\prec$ is transitive.
\end{theorem}
\begin{proof}
For each $s\in \nat$, let $\prec_s$ be the restriction of $\prec$ to $\{0,
\dots, s-1\}$. Notice that $\prec_s$ is the smart reorientation of the
restriction of $\imp$ to $\{0, \dots, s-1\}$. To prove that $\prec$ is
transitive it is enough to check that $\prec_s$ is transitive for each $s$.
We do so by induction on $s$. For the base case there is nothing to check.
Suppose $\prec_s$ is transitive. Then by Property \ref{TriangNonTrans}
$\prec_s$ is lazy. Moreover, by Lemma \ref{PsiPhiThetaSigma}  $\Phi$, $\Psi$
and $\Theta$ are satisfied. Hence, by Lemma \ref{trans} the smart extension
$\prec_{s+1}$ is transitive.
\end{proof}


\end{document}